\documentclass[10pt]{amsart}
\usepackage{a4wide}
\usepackage[utf8]{inputenc}
\usepackage{amsmath, amssymb, bbm}
\usepackage{amsthm}
\usepackage{mathtools}
\usepackage{bbm}
\usepackage{blkarray}
\usepackage[matrix,arrow,curve]{xy}
\usepackage{tikz}
\usepackage{tikz-cd}
\usepackage{graphicx}
\usepackage{color}
\usepackage[normalem]{ulem}
\usepackage{enumerate}
\usepackage{enumitem}
\definecolor{darkblue}{rgb}{0,0,0.6}
\usepackage[ocgcolorlinks,colorlinks=true, citecolor=darkblue, filecolor=darkblue, linkcolor=darkblue, urlcolor=darkblue]{hyperref}
\usepackage[capitalize,noabbrev]{cleveref}
\usepackage{comment}
\usepackage[mathscr]{euscript}
\usepackage{stmaryrd}
\usepackage{tensor}
\usepackage{float}

\newcounter{commentcounter}

\let\oldtocsection=\tocsection

\let\oldtocsubsection=\tocsubsection

\let\oldtocsubsubsection=\tocsubsubsection

\renewcommand{\tocsection}[2]{\hspace{0em}\oldtocsection{#1}{#2}}
\renewcommand{\tocsubsection}[2]{\hspace{1em}\oldtocsubsection{#1}{#2}}
\renewcommand{\tocsubsubsection}[2]{\hspace{2em}\oldtocsubsubsection{#1}{#2}}

\DeclareFontFamily{T1}{cbgreek}{}
\DeclareFontShape{T1}{cbgreek}{m}{n}{<-6>  grmn0500 <6-7> grmn0600 <7-8> grmn0700 <8-9> grmn0800 <9-10> grmn0900 <10-12> grmn1000 <12-17> grmn1200 <17-> grmn1728}{}
\DeclareSymbolFont{quadratics}{T1}{cbgreek}{m}{n}
\DeclareMathSymbol{\qoppa}{\mathord}{quadratics}{19}
\DeclareMathSymbol{\Qoppa}{\mathord}{quadratics}{21}

\newtheorem{Thm}{Theorem}[section]

\newtheorem{Lemma}[Thm]{Lemma}
\newtheorem{Cor}[Thm]{Corollary}
\newtheorem{Prop}[Thm]{Proposition}

\theoremstyle{definition}

\newtheorem{definition}[Thm]{Definition}
\newtheorem{example}[Thm]{Example}
\newtheorem{Rmk}[Thm]{Remark}
\newtheorem{question}[Thm]{Question}
\newtheorem{problem}[Thm]{Problem}
\newtheorem*{ackn}{Acknowledgements}
\numberwithin{equation}{section}
\newtheorem*{intrormk}{Remark}

\theoremstyle{plain}
\newcounter{zaehler}

\newtheorem{introthm}[zaehler]{Theorem}

\newtheorem*{introcor*}{Corollary}

\newcommand{\KK}{\mathrm{KK}}
\newcommand{\adjt}{[\tfrac{1}{2}]}
\renewcommand{\k}{\mathrm{k}}
\newcommand{\K}{\mathrm{K}}
\renewcommand{\L}{\mathrm{L}}
\newcommand{\ko}{\mathrm{ko}}
\newcommand{\ku}{\mathrm{ku}}
\newcommand{\ksp}{\mathrm{ksp}}
\newcommand{\RAlg}{\mathrm{R^*Alg}}
\newcommand{\CAlg}{\mathrm{C^*Alg}}
\newcommand{\Q}{\mathbb{Q}}
\newcommand{\cwedge}{{\scriptscriptstyle\wedge}}

\newcommand{\lto}{\longrightarrow}
\newcommand{\alg}{\mathrm{alg}}
\renewcommand{\top}{\mathrm{top}}

\newcommand{\KU}{\mathrm{KU}}
\newcommand{\KSp}{\mathrm{KSp}}
\newcommand{\KO}{\mathrm{KO}}
\newcommand{\C}{\mathbb{C}}
\newcommand{\R}{\mathbb{R}}
\newcommand{\Z}{\mathbb{Z}}
\newcommand{\E}{\mathbb{E}}
\newcommand{\Gw}{\mathrm{GW}}
\newcommand{\GW}{\mathbb{G}\mathrm{W}}
\newcommand{\Sp}{\mathrm{Sp}}
\newcommand{\Spc}{\mathrm{Spc}}
\newcommand{\map}{\mathrm{map}}
\newcommand{\MSO}{\mathrm{MSO}}
\newcommand{\MSpin}{\mathrm{MSpin}}

\renewcommand{\H}{\mathbb{H}}

\newcommand{\Fun}{\mathrm{Fun}}
\newcommand{\Map}{\mathrm{Map}}

\newcommand{\cofib}{\mathrm{cofib}}
\newcommand{\coker}{\mathrm{coker}}
\newcommand{\op}{\mathrm{op}}
\newcommand{\Unimod}{\mathrm{Unimod}}
\newcommand{\Proj}{\mathrm{Proj}}
\newcommand{\pos}{\mathrm{pos}}
\renewcommand{\neg}{\mathrm{neg}}
\newcommand{\Alg}{\mathrm{Alg}}
\newcommand{\hyp}{\mathrm{hyp}}
\newcommand{\id}{\mathrm{id}}
\newcommand{\free}{\mathrm{free}}

\newcommand{\End}{\mathrm{End}}
\newcommand{\gl}{\mathrm{gl}}
\newcommand{\Ext}{\mathrm{Ext}}
\newcommand{\Hom}{\mathrm{Hom}}
\newcommand{\QF}{\Qoppa}
\newcommand{\s}{\mathrm{s}}
\renewcommand{\SS}{\mathbb{S}}
\newcommand{\Ab}{\mathrm{Ab}}

\newcommand{\sign}{\mathrm{sign}}
\newcommand{\Orb}{\mathrm{Orb}}
\newcommand{\cF}{\mathscr{F}}
\newcommand{\Fin}{\mathscr{F}\mathrm{in}}
\newcommand{\Vcyc}{\mathscr{V}\mathrm{cyc}}
\newcommand{\LL}{\mathbb{L}}
\newcommand{\q}{\mathrm{q}}

\newcommand{\vs}{\mathrm{vs}}
\newcommand{\vq}{\mathrm{vq}}

\newcommand{\Pro}{\mathrm{Pro}}
\newcommand{\cC}{\mathscr{C}}
\newcommand{\D}{\mathscr{D}}
\newcommand{\Ima}{\mathrm{Im}}
\newcommand{\sep}{\mathrm{sep}}
\newcommand{\Ringinv}{\mathrm{Ring}^\mathrm{inv}}
\newcommand{\Catp}{\mathrm{Cat}^\mathrm{p}_\infty}
\newcommand{\nun}{\mathrm{nu}}
\newcommand{\xto}{\xrightarrow}
\newcommand{\BC}{\mathrm{BC}}
\newcommand{\FJ}{\mathrm{FJ}}
\newcommand{\bspin}{\mathrm{bspin}}
\newcommand{\BSpin}{\mathrm{BSpin}}
\newcommand{\Spin}{\mathrm{Spin}}

\DeclareMathOperator{\Eq}{Eq}
\DeclareMathOperator*{\colim}{colim}
\DeclareMathOperator{\can}{can}

\title[L-theory of $C^*$-algebras]{L-theory of $C^*$-algebras}
\date{\today}

\thanks{ML was supported by the CRC/SFB 1085 \emph{Higher Invariants} at the University of Regensburg funded by the DFG and by the Danish National Research Foundation through the Copenhagen Centre for Geometry and Topology (DNRF151). TN was funded by the Deutsche Forschungsgemeinschaft (DFG, German Research Foundation) -- Project-ID 427320536 -- SFB 1442, as well as under Germany's Excellence Strategy EXC 2044 390685587, Mathematics M\"unster: Dynamics--Geometry--Structure.}

\author[M.~Land]{Markus Land}
\address{Mathematisches Institut, Ludwig-Maximilians-Universit\"at M\"unchen, Theresienstra\ss e 39, 80333 M\"unchen, Germany}
\email{markus.land@math.lmu.de}

\author[T.~Nikolaus]{Thomas Nikolaus}
\address{WWU M\"unster, Mathematisches Institut, Einsteinstr. 62, 48149 M\"unster, Germany}
\email{nikolaus@uni-muenster.de}

\author[M.~Schlichting]{Marco Schlichting}
\address{Mathematics Institute, Zeeman Building, University of Warwick, Coventry CV4 7AL, UK} 
\email{m.schlichting@warwick.ac.uk}

\begin{document}
\bibliographystyle{alpha}

\begin{abstract}
We establish a formula for the L-theory spectrum of real $C^*$-algebras from which we deduce a presentation of the L-groups in terms of the topological K-groups, extending all previously known results of this kind. Along the way, we extend the integral comparison map $\tau\colon \k \to \L$ obtained in previous work by the first two authors to real $C^*$-algebras and interpret it using topological Grothendieck--Witt theory. Finally, we use our results to give an integral comparison between the Baum--Connes conjecture and the L-theoretic Farrell--Jones conjecture, and discuss our comparison map $\tau$ in terms of the signature operator on oriented manifolds.
\end{abstract}

\maketitle

\tableofcontents

\section{Introduction}

This paper is concerned with certain invariants of \emph{real} $C^*$-algebras. 
A classical and powerful invariant of real $C^*$-algebras is topological K-theory. However, any $C^*$-algebra is a ring with involution and as such also has an associated (projective) algebraic L-theory. The relation between these two invariants has been an object of investigation for a long time and the purpose of this paper is to give a definitive treatment of this relation.
\newline
 
One of the prominent results in this direction is a theorem due to Karoubi, Miller, and Rosenberg \cite{Karoubi, Miller, Rosenberg2} which states that for complex $C^*$-algebras $A$, there are natural group isomorphisms
\begin{equation}\label{complex_iso}
 \K_n(A) \cong \L_n(A)
 \end{equation}
for all integers $n$. It is, however, well-known that the (topological) K-theory spectrum $\K(A)$ and the (algebraic) $\L$-theory spectrum $\L(A)$ are \emph{not} equivalent and that the isomorphism \eqref{complex_iso} does not hold true for real $C^*$-algebras $A$.
In previous work of the first two authors, the relation between the topological K- and L-spectra of complex $C^*$-algebras was studied \cite{LN}. Neglecting 2-torsion, or more precisely after inverting 2, this relation extends to real $C^*$-algebras and one can summarise the situation as follows: On complex $C^*$-algebras, there is a unique lax symmetric monoidal natural transformation $\tau\colon \k \to \L$ which induces an equivalence $\K\adjt \to \L\adjt$, and this latter equivalence extends in a compatible way to real $C^*$-algebras. Here, $\k$ denotes the connective topological $\K$-theory spectrum functor, i.e.\ the connective cover of $\K$. 
The map $\tau_A \colon \k(A) \to \L(A)$ induces an isomorphism on $\pi_0$ and $\pi_1$, so that by 2-periodicity of the two theories, one recovers the fact that all L-groups are isomorphic to the corresponding topological K-groups. However, under this isomorphism the map $\tau_\C \colon \k(\C) \to \L(\C)$ induces multiplication by 2 on $\pi_2$, so integrally, the periodicity in K-theory does not match up with the periodicity in L-theory. Explicitly left open in \cite{LN} was an integral comparison between K- and L-theory for real $C^*$-algebras, a gap which will be reconciled in this paper.
\newline

For the rest of this paper, $C^*$-algebras are now agreed to be real $C^*$-algebras, we will add the adjective complex when we need it. The purpose of this paper is to explain in full generality how to describe the L-theory of $C^*$-algebras in terms of their topological K-theory and in particular how to express the L-groups in terms of topological K-groups. We emphasize that L-theory refers to \emph{projective} L-theory. We will also discuss free L-theory of unital algebras in \cref{sec:5} but stick to the case of projective $\L$-theory for this introduction.
We note that both K- and L-theory of a complex $C^*$-algebra depend only on the underlying (real) $C^*$-algebra, so the case of complex $C^*$-algebras is treated implicitly. To follow standard notation in homotopy theory, we shall write $\KO$ and $\KU$ for $\K(\R)$ and $\K(\C)$, respectively, and likewise $\ko$ and $\ku$ for their connective covers $\k(\R)$ and $\k(\C)$, respectively.
The following are our main results. 

\begin{introthm}\label[Thm]{ThmA}
There is a unique lax symmetric monoidal transformation $\tau \colon \k \to \L$ and the induced map 
\[ \k(A) \otimes_\ko \L(\R) \lto \L(A) \]
is an equivalence of spectra for each $C^*$-algebra $A$.
\end{introthm}

\begin{introthm}\label[Thm]{ThmB}
Let $A$ be a $C^*$-algebra. There are natural isomorphisms of abelian groups
\begin{enumerate}
\item $\L_0(A) \cong \K_0(A)$, 
\item $\L_1(A) \cong  \coker(\K_0(A) \xrightarrow{\cdot \eta} \K_1(A))$
\item $\L_{2}(A) \cong \ker(\K_{6}(A) \xrightarrow{\cdot \eta} \K_{7}(A)  )$
\item $\L_3(A) \cong \K_7(A)$.
\end{enumerate}
Here $\eta$ is the non-trivial element in $\K_1(\R) = \pi_1(\KO)$.
\end{introthm}

\begin{intrormk}
For a $C^*$-algebra $A$, there is the generalised Wood exact sequence
\[\cdots  \lto \K_{n-1}(A) \stackrel{ \eta}{\lto} \K_{n}(A) \stackrel{c}{\lto} \K_n(A_\C) \stackrel{u\beta_\C^{-1}}{\lto} \K_{n-2}(A) \stackrel{ \eta}{\lto} \K_{n-1}(A) \lto \cdots \]
Consequently, we also find canonical isomorphisms
\begin{enumerate} 
	\item $\L_1(A) \cong \ker\big( \K_{-1}(A_\C) \stackrel{u}{\to} \K_{-1}(A) \big)$, and
	\item $\L_2(A) \cong \coker\big( \K_0(A) \stackrel{c}{\to} \K_0(A_\C) \big)$.
\end{enumerate}
\end{intrormk}

By the $4$-fold periodicity of L-theory, \cref{ThmB} gives a natural description of all L-groups in terms of topological K-groups, see \cref{Thm:natural-ThmB} for a more canonical formulation. Under these isomorphisms we also describe the effect on homotopy groups of the map $\tau_A \colon \K_n(A) \to \L_n(A)$ from Theorem \ref{ThmA} for $n\geq 0$ (see \cref{Rem:induced-map}) as well as the exterior multiplication maps on the L-groups in terms of the exterior multiplication maps on the K-groups (see \cref{Rem:multiplication}).
In \cref{Sec:examples} we discuss a number of examples and calculate L-groups using \cref{ThmB}.

\begin{intrormk}
We note that \cref{ThmA} implies that two real $C^*$-algebras $A_0$ and $A_1$ whose K-theory spectra are equivalent as module spectra over $\ko$ have equivalent L-theory spectra. Likewise, \cref{ThmB} implies that if the K-groups of $A_0$ and $A_1$ are isomorphic as graded $\Z[\eta]$-modules, where $|\eta|=1$ and $\eta$ acts in the canonical way on the K-groups, then also the L-groups of $A_0$ and $A_1$ are isomorphic.
\end{intrormk}

Kasparov's KK-theory is a central tool for studying the K-theory of $C^*$-algebras. It therefore comes as no surprise that one would also like to study L-theory of $C^*$-algebras by KK-theoretic means. In the case of complex $C^*$-algebras, this was done in \cite{LN} but at the time of writing \cite{LN}, it was not known whether L-theory of $C^*$-algebras is KK-invariant in general. Even the fact that it is KK-invariant on complex $C^*$-algebras is a result which we still find quite surprising, since KK-theory is an intrinsically analytic theory, whereas L-theory depends only on the underlying algebraic structure of a $C^*$-algebra. Even more, $\L$-theory commutes with filtered colimits of involutive rings and thus only depends on the underlying algebraic structure of proper involutive subalgebras, which do not themselves need to be $C^*$-algebras.

It is an immediate consequence of \cref{ThmA} that L-theory is a KK-invariant functor on $C^*$-algebras. Our proof works the other way around though: instead of deducing KK-invariance from \cref{ThmA}, we use it as an input for the proof of \cref{ThmA}, and we give an argument for KK-invariance based on a description for 2-complete L-theory instead. Indeed, $\L\adjt$ was shown to be KK-invariant in \cite{LN} so it remains only to see that $\L(-)^\cwedge_2$ is KK-invariant. This is a direct consequence of the following result, which we derive from \cite[Theorem D.1]{KSW}.

\begin{introthm}\label[Thm]{ThmC}
For every Banach algebra with involution, the canonical map $\L(A) \to \k(A)^{tC_2}$ is a 2-adic equivalence.
\end{introthm}

This is a topological version of Thomason's homotopy limit problem in hermitian algebraic $\K$-theory. This algebraic homotopy limit problem has been studied extensively, see e.g.\ \cite{HKO,BH,BKSO,CDHIII} for the case of fields, schemes over $\Z\adjt$, and Dedekind rings.

\subsection*{Assembly Maps}
As in \cite{LN}, a major motivation for studying the relation between K- and L-theory of $C^*$-algebras is to obtain a precise relationship between the Baum--Connes conjecture and the L-theoretic Farrell--Jones conjecture, inspired by the observation that both of these conjectures imply the Novikov conjecture. In \cite{LN} such a relation was understood after inverting 2 and we offer here the following integral refinement:

\begin{introthm}\label[Thm]{ThmD}
The map $\tau \colon \k \to \L$ induces a commutative diagram
\[\begin{tikzcd}
	\ko_*^G(\underline{E}G) \ar[rr,"\mathrm{BC}"] \ar[d,"\tau"] & & \k_*(C^*_rG) \ar[d,"\tau"] \\
	\L\R^G_*(\underline{\underline{E}}G) \ar[r,"\mathrm{FJ}"] & \L_*(\R G) \ar[r] & \L_*(C^*_r G)
\end{tikzcd}\]
where BC and FJ denote the Baum--Connes and Farrell--Jones assembly maps.
\end{introthm}
After inverting the Bott element $\beta_\R$ and $2$, one recovers \cite[Theorem D]{LN}. In addition, the kernel and cokernels of the vertical maps can in principle be described using our identification of $\tau$ on homotopy groups. For the left hand vertical map this is most effective in the case where $G$ is torsion free as explained in Section \ref{Sec_seven}.
\newline

Before our work \cite{LN}, there have already been made several fruitful efforts to relate the surgery theoretic and the analytic approach to the Novikov conjecture, most notably the work of Higson and Roe \cite{HR1,HR2,HR3}. There, a central idea is to consider the signature operator $D_M$ of an oriented manifold $M$ as an appropriate K-theory class and use this to construct a comparison map from the surgery exact sequence to a 2-inverted exact sequence of topological K-groups. It has been known for a long time that the signature operator of an oriented manifold, unlike the spin Dirac operator, does not give rise to a map of spectra $\MSO \to \ko$, due to factors of 2 appearing for the signature operator on a boundary\footnote{Also, any map between those spectra induces the trivial map on homotopy groups as follows from the fact that $\MSO$ at primes 2 vanishes $K(1)$-locally.}, see \cite[Remark 4]{RW}. The following theorem expresses the fact that, with appropriate modifications, the signature operator does give rise to an $\E_\infty$ map $\MSO \to \ko\adjt$ and clarifies its relation with the Sullivan--Ranicki orientation; a version of this theorem discarding $\E_\infty$-structures was discussed in \cite{RW}.

\begin{introthm}
The association $M^{2n} \mapsto 2^{- \lfloor n/2 \rfloor} \cdot [D_M]$ refines uniquely to a  map of $\E_\infty$-ring spectra $\mathscr{L}_{AS} \colon \MSO \to \ko\adjt$. This map participates in the following commutative diagram of $\E_\infty$-rings
\[ \begin{tikzcd}
	\MSO \ar[r,"{\mathscr{L}_{AS}}"] \ar[d,"\sigma_\R"] & \ko\adjt \ar[d,"\tau"] \\
	\L(\R) \ar[r,"\mathrm{can}"] & \L(\R)\adjt
\end{tikzcd}\]
where $\sigma_\R$ is the Sullivan--Ranicki orientation.
\end{introthm}

Here, the map $\mathscr{L}_{AS}$ induces on homotopy groups a version of the L-genus, precisely the version of the L-genus that has been employed by Atiyah and Singer in their index-theoretic proof of Hirzebruch's signature theorem \cite{MR236952}. The map $\sigma_\R$ on the other hand induces Hirzebruch's original $\L$-genus. Thus the result says that the two differ exactly by our comparison map.

\begin{ackn}
The authors would like to thank Johannes Sprang for his explanations regarding $p$-adic moment sequences and Johannes Ebert, Michael Joachim, and Achim Krause for helpful discussions, as well as anonymous referees for their helpful comments.
The authors would also like to thank the Hausdorff Center of mathematics for hospitality and providing a great working environment during the conference ``Hermitian K-theory and trace methods'' in November 2016. ML gave a talk about the complex case of \cref{ThmC} building on the earlier results obtained with TN and it was there that this paper was born. 
\end{ackn}

\section{Preliminaries}
In this section we will briefly recall the notions of $C^*$-algebras, KK-theory, and L-theory.
\subsection*{$C^*$-algebras}
A nice reference for $C^*$-algebras over $\R$ and their K-theory is Schroeder's book \cite{Schroeder}, further references include \cite{Constantinescu, Goodearl, Li} and \cite{Takesaki1, Takesaki2, Takesaki3} for complex $C^*$-algebras.
\begin{definition}
A $C^*$-algebra is Banach algebra $A$ equipped with an involution $(-)^*\colon A \to A^\op$ with $x^{**} = x$ such that the following two conditions hold:
\begin{enumerate}
\item For all $x \in A$ we have $|| x^*x|| = ||x||^2$, and
\item for all $x \in A$, the element $1+x^*x$ is invertible in the unitalization $A^+$.
\end{enumerate}
A complex $C^*$-algebra is a complex Banach algebra $A$ whose underlying real Banach algebra is a $C^*$-algebra and where the involution is complex sesquilinear, i.e.\ $(\lambda x)^* = \bar{\lambda} x^*$.
\end{definition}

\begin{Rmk}
The condition that $1+x^*x$ is invertible might be a bit surprising at first glance. We note that it is a consequence of spectral calculus that this condition is automatically fulfilled for complex $C^*$-algebras, see e.g.\ \cite{Takesaki1}. In the real case it can however not be left away, since for example $\C$ equipped with the identity involution satisfies the other conditions but $1+i^2=0$ is not invertible.
\end{Rmk}

\begin{Rmk}
The well-known structure theorems for complex $C^*$-algebras have the following real analogues:
\begin{enumerate}
\item Every $C^*$-algebra has a faithful representation on a real Hilbert space $\mathcal{H}$, i.e.\ is isometrically isomorphic to a $*$- and norm-closed subalgebra of $\mathcal{B}(\mathcal{H})$.
\item Every commutative and unital $C^*$-algebra is isometrically isomorphic to the $C^*$-algebra of $C_2$-equivariant continuous functions $X \to \C$ for a compact Hausdorff space $X$ equipped with a $C_2$-action, where $C_2$ acts on $\C$ by complex conjugation.
\end{enumerate}
\end{Rmk}

\begin{Rmk}
A number of remarks are in order.
\begin{enumerate}
	\item Together with $*$-homomorphisms, $C^*$-algebras form a category $\CAlg$, and likewise complex $C^*$-algebras form a category $\CAlg_\C$. We emphasise that $C^*$-algebras are not assumed to be unital, nor that $*$-homomorphisms are assumed to preserve a unit if it exists. Requiring, however, algebras to have a unit and morphisms to preserve units, one obtains similarly the categories $\CAlg^+$ and $\CAlg^+_\C$ of unital $C^*$-algebras.
	\item We note that $*$-homomorphisms are automatically contractive and hence continuous.
	\item By construction, there is a forgetful functor $\CAlg_\C \to \CAlg$ which we call the \emph{realification}, moreover, the construction $A \mapsto A_\C \stackrel{\mathrm{def}}{=} A\otimes_\R \C$ extends to a natural functor $\CAlg \to \CAlg_\C$, which we call the \emph{complexification}.
	\item There are unitalisation functors $(-)_\R^+\colon \CAlg \to \CAlg^+$ and $(-)_\C^+\colon \CAlg_\C \to \CAlg_\C^+$ which come with natural split exact sequences 
	\[ 0 \lto A \lto A^+_{\R} \lto \R \lto 0 \quad \text{ and } \quad 0 \lto B \lto B_{\C}^+ \lto \C \lto 0 \]
	respectively. If $A$ is unital, then $A^+_{\R}$ is canonically isomorphic to $A \times \R$, and likewise in the complex case.
\item The complexification functor is compatible with unitalisation, whereas the realification functor is not compatible with unitalisation. More precisely the solid diagram commutes, whereas the diagram involving dashed arrows does not.
\[\begin{tikzcd}
	\CAlg \ar[r] \ar[d] & \CAlg^+ \ar[d] \\
	\CAlg_\C \ar[r] \ar[u, bend left, dashed] & \CAlg_\C^+ \ar[u, bend left, dashed]
\end{tikzcd}\]
\item The categories $\CAlg$ and $\CAlg_\C$ are each equipped with a canonical symmetric monoidal structure, the maximal tensor product over $\R$ and $\C$, respectively. The maximal tensor product preserves short exact sequences of $C^*$-algebras and topological $\K$-theory is canonically lax symmetric monoidal.
\end{enumerate}
\end{Rmk}

\begin{definition}
A $C^*$-algebra is called separable if it contains a countable and dense subset. The full subcategory of $\CAlg_{(\C)}$ on separable $C^*$-algebras will be written $\CAlg^\sep_{(\C)}$. 
\end{definition}
The complexification and realification functors restrict to the subcategory of separable algebras. In addition, we note that every $C^*$-algebra is the union of its separable $C^*$-subalgebras and that the collection of separable $C^*$-subalgebras forms a filtered poset. For technical reasons, we will restrict our attention to separable algebras momentarily. However, all invariants $F$ of $C^*$-algebras we shall consider (i.e.\ topological K-theory and L-theory) send an algebra $A$ to the filtered colimit of $F$ applied to the separable subalgebras of $A$, and consequently, we can get rid of the separability assumptions.

\subsection{KK-theory}
In his seminal work on the Novikov conjecture \cite{Kasparov}, Kasparov invented (equivariant) bivariant topological K-theory, known as KK-theory. Phrased in categorical language, Kasparov's machine allowed to construct a tensor triangulated category $\KK$ and a functor
\[ \CAlg^\sep \lto \KK \]
which was later shown to be a localisation (necessarily at the KK-equivalences, i.e.\ those $*$-homomorphisms whose induced map in the KK-category is an isomorphism) \cite{Cuntz} and to be the initial functor to an additive category which is split exact and stable \cite{Higson}, see e.g.\ \cite{BEL} for more precise statements and a guide through (parts of) the literature. 
In \cite{LN}, it was then observed that the $\infty$-categorical localisation of $\CAlg^\sep$ at the KK-equivalences is a stably symmetric monoidal $\infty$-category whose homotopy category is canonically equivalent to the tensor triangulated category KK of Kasparov. This observation has also been taken up in \cite{BEL} (including extensions of these results to possibly non-separable $C^*$-algebras) in the equivariant case and was used in \cite{BEL2} in a proof of an equivariant form of Paschke duality.

\begin{definition}
We denote by $\KK = \CAlg^\sep[\KK^{-1}]$ the $\infty$-categorical localisation of $\CAlg^\sep$ at the KK-equivalences. Likewise, we denote by $\KK_\C = \CAlg^\sep_\C[\KK^{-1}]$ the variant for complex $C^*$-algebras.
\end{definition}

\begin{Rmk}
In \cite{LN}, different notation was used: In loc.\ cit.\, the authors were focussed mostly on the complex case and therefore denoted $\CAlg^\sep_\C[\KK^{-1}]$ by $\KK_\infty$, and its real variant by $\KK_\infty^\R$; the subscript $\infty$ was added to make clear that one was now working with an appropriate $\infty$-category rather than a triangulated category. We refrain from adding this subscript in this paper, however.
\end{Rmk}

\begin{definition}
The topological K-theory functor for separable $C^*$-algebras is given by the composite
\[ \K \colon \CAlg^\sep \lto \KK \lto \Sp \]
where the first functor is the localisation functor, the second is the corepresented functor $\map_\KK(\R,-)$, and $\Sp$ denotes the $\infty$-category of spectra.
\end{definition}

\begin{Rmk}
There are of course other, more classical definitions of topological K-theory functors \cite{Joachim1, Joachim2}, and it was shown in \cite{LN} that they are canonically equivalent to the definition given above. These more classical definitions are in fact given for possibly non-separable algebras and satisfy
\[ \K(A) \simeq \colim\limits_{A' \subseteq_{\sep} A} \K(A') \]
so we may also view the above definition as describing K-theory of possibly non-separable algebras. In \cite{BEL} this was formalised by considering the ind-completion of $\KK$ and again considering the functor corepresented by $\R$.

This definition, however, does not give all structure that topological K-theory has: For instance, it is a purely formal consequence of the definitions that K-theory sends certain short exact sequences (e.g.\ where the surjection is a Schochet fibration or admits a cpc split) to fibre sequences, but it is not a priori clear that it sends all short exact sequences of $C^*$-algebras to fibre sequences. However, this is known to be true, e.g. as a special case of \ \cite[Theorem {1.15}]{BEL} (for $X=\ast$ in the notation of loc. cit.).
\end{Rmk}

\subsection{L-theory}
In this subsection, we review some basic properties of L-theory which we will use throughout this paper, see also \cite[\S 2.2]{LN} for a further summary.

For our purposes, L-theory is most naturally considered as a functor introduced by Ranicki in \cite{Ranickiblue}
\[ \Ringinv \lto \Sp \]
where $\Ringinv$ is the category of involutive rings with ring homomorphisms preserving the involution. In fact, this functor can be written as the composition
\[ \Ringinv \lto \Catp \lto \Sp \]
where $\Catp$ is the $\infty$-category of Poincar\'e categories on which L-theory is a natural invariant, see \cite{CDHI,CDHII,CDHIII} for applications of this formalism to Grothendieck--Witt theory of number rings. Together with \cite{CDHIV}, or using results of Laures--McClure \cite{LMcC1, LMcC2}, L-theory is canonically endowed with a lax symmetric monoidal structure.
The first functor in the above composite sends a ring with involution $R$ to the pair $(\D^p(R),\QF^\s)$, so more precisely we are considering projective, 4-periodic symmetric L-theory of involutive rings in the sense of \cite{Ranickiblue}. Prior to the work \cite{CDHI,CDHII,CDHIII}, the third author had introduced L-spectra for dg-categories over $\Z\adjt$ with weak equivalences \cite[\S 7]{Schlichting}. In \cite[Appendix B.2]{CDHII}, it is shown that for $\Z\adjt$-algebras with involution, the two constructions of L-spectra are naturally equivalent.

There are natural forgetful functors 
\[ \CAlg^+ \lto \Ringinv_{\Z\adjt} \lto \Ringinv \]
which define L-theory of \emph{unital} $C^*$-algebras. Since many of the possibly different notions of L-theory agree on rings in which $2$ is invertible, and since this paper is concerned with $C^*$-algebras, we shall from now on restrict our attention to $\Z\adjt$-algebras with involution as the domain of L-theory.

\subsection*{L-theory for non-unital algebras}
For our applications, which involve KK-theory, it is necessary to define L-theory for possibly non-unital algebras. For this we define a unitalisation of non-unital rings in the usual way
\[ \Ringinv_{\Z\adjt,\nun} \lto \Ringinv_{\Z\adjt} \quad R \mapsto R^+_{\Z\adjt} \]
and note again that for unital $\Z\adjt$-algebras $S$, we have $S^+_{\Z\adjt} \cong S \times \Z\adjt$.
We then define L-theory on non-unital $\Z\adjt$-algebras as follows
\[ \L(R) \stackrel{\mathrm{def}}{=} \mathrm{fib} \big( \L(R^+_{\Z\adjt}) \to \L(\Z\adjt)\big).\]

Since L-theory commutes with finite products \cite[Corollary 4.4]{LN} (for this to be true it is crucial to work with \emph{projective} L-theory, rather than free L-theory which appears in the h-cobordisn classification program in surgery theory), we have not changed the definition of L-theory on unital rings, up to canonical equivalence.

However, from the point of view of applying L-theory to $C^*$-algebras, we now have constructed two functors 
\[ \CAlg \lto \Ringinv \]
one given by $A \mapsto A^+_{\R} $ and the other one given by $A \mapsto A^+_{\Z\adjt} $, i.e.\ we can either unitalise in $\R$-algebras or in $\Z\adjt$-algebras. Moreover, for complex algebras, we have three such functors, by adjoining a unit in $\C$-algebras, $\R$-algebras, or $\Z\adjt$-algebras, respectively\footnote{Of course, one could also unitalise in $\Z$-algebras, but see \cref{Rmk:excision-symmetric-L} below.}. We note that for a $C^*$-algebra $A$, there is a natural pullback diagram
\[ \begin{tikzcd}
	A^+_{\Z\adjt} \ar[r] \ar[d] & A^+_\R \ar[d] \ar[r] & A^+_\C \ar[d] \\
	\Z\adjt \ar[r] & \R \ar[r] & \C
\end{tikzcd}\]
where the right most vertical part only exists if $A$ is a complex $C^*$-algebra and where the vertical maps are split surjective. It is a theorem of Ranicki \cite{Ranickiyellow}, see e.g.\ \cite[Corollay 4.3]{LN} that both squares induce pullback squares on L-theory. Consequently, extending L-theory to non-unital (complex) $C^*$-algebras can be performed either by adjoining a unit in $\C$-algebras, or by forgetting to the underlying real $C^*$-algebra and then adjoining a unit in $\R$-algebras, or by forgetting to the underlying $\Z\adjt$-algebra and adjoining a unit there. 

\begin{Rmk}\label[Rmk]{Rmk:excision-symmetric-L}
We do not expect the diagram
\[ \begin{tikzcd}
	\L(A^+_\Z) \ar[r] \ar[d] & \L(A^+_{\Z\adjt}) \ar[d] \\
	\L(\Z) \ar[r] & \L(\Z\adjt)
\end{tikzcd}\]
to be a pullback for every $\Z\adjt$-algebra $A$. As a consequence, we do not expect the definition of $\L$-theory for non-unital rings to be independent of the base over which the unitalisation is performed in general. Ranicki however shows that this square is a pullback if symmetric L-theory is replaced by quadratic L-theory, see \cite[6.3.1]{Ranickiyellow}.
\end{Rmk}

Finally, as explained in \cite[Appendix]{LN}, the fact that L-theory is lax symmetric monoidal on unital $C^*$-algebras allows to deduce that L-theory as defined above is in fact canonically lax symmetric monoidal on all $C^*$-algebras.

\section{Proof of \cref{ThmC}}

For convenience we state again the theorem we shall prove in this section. We emphasize that KK-theory is not used in this proof.
\begin{Thm}
Let $A$ be a Banach algebra with involution. Then the canonical map $\L(A) \to \k(A)^{tC_2}$ is a 2-adic equivalence.
\end{Thm}
In the proof of this theorem and in fact also of \cref{ThmA}, we will make use of the topological Grothendieck--Witt spectra introduced in \cite[\S 10]{Schlichting} which we denote by $\Gw_\top(A)$. We recall that the family of topological $n$-simplices $\{\Delta^n\}_{n\in \Delta}$ form, via the canonical coface and codegeneracy maps, a cosimplicial topological space. Hence, by considering algebras of continuous functions one obtains a simplicial involutive Banach algebra $C^0(\Delta^n,A)$ (with pointwise involution) and one defines $\Gw_\top(A)$
as the geometric realisation
\[ \Gw_\top(A) = \colim\limits_{n \in \Delta^\op} \Gw(C^0(\Delta^n,A)) \]
of the resulting simplicial spectrum, where $\Gw$ is the (algebraic) Grothendieck--Witt functor. 
We recall from \cite[Prop.\ 10.2]{Schlichting} that connective topological K-theory $\k(A)$ admits a similar description in terms of connective algebraic K-theory $\K_\alg$:
\[ \k(A) = \colim\limits_{n \in \Delta^\op} \K_{\alg}(C^0(\Delta^n,A)).\]

For any ring with involution $R$ with $2 \in R^\times$ there is a natural fibre sequence
\begin{equation}\label{fundamental-sequence} (\K_\alg(R))_{hC_2} \lto \Gw(R) \lto \L(R),\end{equation}
see e.g.\ \cite[Main Theorem]{CDHII} or \cite[Theorem 7.6]{Schlichting} using that the Grothendieck--Witt spectra of \cite{CDHII} and of \cite{Schlichting} agree, see \cite[Appendix B.2]{CDHII}. More specifically, in the notation of \cite{CDHII}, $\Gw(R)$ is given by $\Gw(R;\QF^\s_R)$, and likewise $\L(R)$ is given by $\L(R;\QF^\s_R)$; we remark here that by assumption $2$ is invertible in $R$, many of the a priori different versions of Grothendieck--Witt theory studied in \cite{CDHII, CDHIII} collapse to the same object, which we here simply denote by $\Gw(R)$, see e.g.\ \cite[Remark R.4]{CDHIII}.
By \cite[Corollary 4.4.14]{CDHII} this fibre sequence can also be encoded in the following natural pullback diagram.
\begin{equation}\label{square}
 \begin{tikzcd}
	\Gw(R) \ar[r] \ar[d] & \L(R) \ar[d] \\
	\K_\alg(R)^{hC_2} \ar[r] & \K_\alg(R)^{tC_2} 
\end{tikzcd}
\end{equation}
We may then likewise define 
\[ \L_\top(A) = \colim\limits_{n \in \Delta^\op} \L(C^0(\Delta^n,A)).\]
An astonishing feature of algebraic L-theory is the following homotopy invariance statement.
\begin{Prop}\label[Prop]{Prop:homotopy-invariance-of-L}
For every Banach algebra with involution, the canonical map $\L(A) \to \L_\top(A)$ is an equivalence.
\end{Prop}
\begin{proof}

Since $\L$-theory is $4$-periodic, it suffices to show that $\L(A) \to \L_\top(A)$ induces an isomorphism on negative homotopy groups. Recall that a sequence of spectra is a fibre sequence if and only if it is a cofibre sequence. Since colimits commute with colimits, we find that geometric realizations preserve cofibre sequences and homotopy orbits for group actions. Consequently, from the fibre sequence \eqref{fundamental-sequence} and the definitions, we have a fibre sequence
\begin{equation} \label{fundamental-sequence-top}
	\k(A)_{hC_2} \lto \Gw_\top(A) \lto \L_\top(A).
\end{equation}
Since the first term in this sequence is connective, the latter map induces isomorphisms on negative homotopy groups. \cite[Remark 10.4]{Schlichting} gives that $\Gw(A) \to \Gw_\top(A)$ induces an isomorphism on negative homotopy groups, so the proposition is proven.
\end{proof}

Combining \cref{Prop:homotopy-invariance-of-L} with the fibre sequence \eqref{fundamental-sequence-top}, we obtain the following corollary.
\begin{Cor}\label[Cor]{Cor:fibre-sequence-top}
There is a natural fibre sequence of functors 
\[ \k(-)_{hC_2} \lto \Gw_\top(-) \lto \L(-).\]
\end{Cor}

\begin{proof}[Proof of \cref{ThmC}]
We first claim that there is the following natural square of spectra
\begin{equation}\label{pullbackL}
 \begin{tikzcd}
	\Gw_\top(A) \ar[r] \ar[d] & \L(A) \ar[d] \\
	\k(A)^{hC_2} \ar[r] & \k(A)^{tC_2}
\end{tikzcd}
\end{equation}
and that this square is a pullback square.
To see this we use \cref{Prop:homotopy-invariance-of-L} to replace $\L(A)$ by its topological variant and by taking the geometric realization of \eqref{square} we get a diagram as desired (using the canonical colimit interchange map for the lower two corners). To see that it is a pullback we use that the canonical map induced on horizontal fibres is an equivalence, since homotopy orbits commute with the geometric realization.
We wish to show that the right vertical map is an equivalence modulo 2. To do so, we first note that since the transformation $\L(-) \to \k(-)^{tC_2}$ is lax symmetric monoidal, the map $\L(A) \to \k(A)^{tC_2}$ is one of $\L(\R)$-modules so its fibre is 4-periodic. Then we consider the Bott-periodic analog $\GW_\top(A)$ of $\Gw_\top(A)$ also used in \cite[Appendix D]{KSW} which participates in the following commutative diagram.
\[\begin{tikzcd}
	\k(A)_{hC_2} \ar[r] \ar[d] & \Gw_\top(A) \ar[d] \\
	\K(A)_{hC_2} \ar[r] & \GW_\top(A)
\end{tikzcd}\]
One defines a functor $\mathbb{L}_\top$\footnote{We warn the reader that this object is not given by the topological version of Karoubi-invariant L-theory sometimes denoted by $\mathbb{L}$ in the literature.} as the cofibre of the lower horizontal map in the above square. By \cref{Cor:fibre-sequence-top} one obtains a map from the left to the right following square
\begin{equation}\label{pullbackBig}
 \begin{tikzcd}
	\Gw_\top(A) \ar[r] \ar[d] & \GW_\top(A) \ar[d] & \L(A) \ar[r] \ar[d] & \mathbb{L}_\top(A) \ar[d] \\
	\k(A)^{hC_2} \ar[r] & \K(A)^{hC_2} & \k(A)^{tC_2} \ar[r] & \K(A)^{tC_2} 
\end{tikzcd}
\end{equation}
again using the definition of the Tate construction. One directly checks that the square of fibres is cartesian, so that the induced map of total fibres of the above squares is an equivalence.
Furthermore, from \cite[Lemma D.2]{KSW}, it follows that $\mathbb{L}_{\top}(A)/2 = 0 = \K(A)^{tC_2}/2$. Consequently, modulo 2 there is a canonical equivalence between the total fibre of the left square in \eqref{pullbackBig} and the fibre of $\L(A) \to \k(A)^{tC_2}$. Now, the top horizontal map of the left square in \eqref{pullbackBig} induces an equivalence on connective covers, see \cite[Proof of Theorem D.1]{KSW}, so its fibre is coconnected. Likewise, the bottom horizontal fibre is 
$(\tau_{<0}\K(A))^{hC_2}$ which is coconnected since coconnected spectra are closed under limits in spectra. Hence, the total fibre of the left square in \eqref{pullbackBig} is coconnected. In total, it follows that the fibre modulo 2 of the map $\L(A) \to \k(A)^{tC_2}$  is bounded above and 4-periodic, hence trivial.\footnote{In fact, the total fibre is 4-periodic integrally: This is because diagram involving $\mathbb{L}$, $\mathbb{L}_\top$, $\k(A)^{tC_2}$ and $\K(A)^{tC_2}$ is one of $\L(\R)$-modules, but for sake of shortness, we omit an argument here.} 
\cref{ThmC} is therefore proven.
\end{proof}

\begin{Cor}\label[Cor]{Cor:KK-invariance}
The functor $\L$ descends to a functor $\KK \to \Sp$.
\end{Cor}
\begin{proof}
The arithmetic fracture square provides a pullback square of functors 
\[ \begin{tikzcd}
	\L \ar[r] \ar[d] & \L^\cwedge_2 \ar[d] \\
	\L\adjt \ar[r] & \L^\cwedge_2\adjt
\end{tikzcd}\]
so it suffices to show that $\L\adjt$ and $\L^\cwedge_2$ are KK-invariant. The former is a direct consequence of the natural isomorphism between $\K_n(-)\adjt$ and $\L_n(-)\adjt$ (see also \cite{LN}) and the latter follows from \cref{ThmC}.
\end{proof}

\begin{Rmk}\label[Rmk]{remark_yoneda}
In what follows, we will crucially use different variants of the $\infty$-categorical Yoneda lemma. We will make these different variants explicit now. To that end, assume that $\cC$ is a stable $\infty$-category, the relevant example for this paper is $\cC = \KK$, and let $c \in \cC$ be an object. Then 
the \emph{mapping space} functor $\Map_{\cC}(c, -)\colon \cC \to \Spc$ admits an essentially unique refinement to a limit preserving functor $\map_{\cC}(c,-)\colon \cC \to \Sp$. Here, refinement means that it is equipped with a natural equivalence $\Omega^\infty\map_{\cC}(c,-) \simeq \Map_{\cC}(c, -) $. The spectrum $\map_{\cC}(c,d)$ is called the \emph{mapping spectrum} in $\cC$, it recovers the mapping space upon applying $\Omega^{\infty}$. 

Now we have the following versions of the Yoneda lemma (see e.g.\ \cite[Lemma 3.6]{LN} for references and proofs of \eqref{yoneda-item1}--\eqref{yoneda-item3} and \cite{Nikolaus} for \eqref{yoneda-item4}). In all instances, the natural equivalences are induced by evaluating a natural transformation on the identity.
\begin{enumerate}
\item\label{yoneda-item1} For any functor $F\colon \cC \to \Spc$, we have a natural equivalence
\[
\Map_{\Fun(\cC, \Spc)}(\Map_{\cC}(c, -), F) \simeq F(c).
\] 
\item\label{yoneda-item2}
For any finite product preserving functor $F\colon \cC \to \Sp_{\geq 0}$, we have a natural equivalence
\[
\Map_{\Fun(\cC,\Sp_{\geq 0})}(\tau_{\geq 0}\map_{\cC}(c, -), F) \simeq \Omega^\infty F(c).
\] 
\item \label{yoneda-item3}
For any finite product preserving functor 
$F\colon \cC \to \Ab$, we have a natural isomorphism
\[
\Hom_{\Fun(\cC,\Ab)}(\pi_0\map_{\cC}(c, -), F) \simeq \pi_0F(c).
\] 
\item\label{yoneda-item4} 
If $\cC$ is equipped with a symmetric monoidal structure, $F\colon \cC \to \Sp_{\geq 0}$ is lax symmetric monoidal and preserves finite products and $\mathbbm{1}$ is the tensor unit, then the space 
 \[
\Map_{\Fun^{\mathrm{lax}}(\cC, \Sp_{\geq 0})}(\tau_{\geq 0}\map_{\cC}(\mathbbm{1}, -), F)
 \]
 of lax symmetric monoidal transformations is contractible, i.e.\ there is a unique lax symmetric monoidal transformation $\tau_{\geq0}\map_{\cC}(\mathbbm{1}, -) \to F$. Under the identification of \eqref{yoneda-item2} this corresponds to the element $1$ in the algebra $\pi_0(F(\mathbbm{1}))$.
\end{enumerate}
In particular from \eqref{yoneda-item2} and \eqref{yoneda-item3} together, we deduce that for any finite product preserving functor $F\colon \cC \to \Sp_{\geq 0}$ we have a bijection
between homotopy classes of natural transformations $\tau_{\geq 0}\map_{\cC}(c, -) \to F$ and natural transformations $\pi_0 \map_{\cC}(c, -) \to \pi_0 F$. This bijection is given by taking the effect on $\pi_0$ of a natural transformation, as follows from the explicit description of the equivalences in \eqref{yoneda-item2} and \eqref{yoneda-item3}. In other words: every transformation $\pi_0 \map_{\cC}(c, -) \to \pi_0 F$ can be uniquely (up to homotopy) extended to a transformation of connective spectrum valued functors. 
In fact, if $F$ were finite limit preserving then one could even extend to a natural transformation of spectrum valued functors, by the following stable version of the Yoneda lemma, which says that for any finite limit preserving functor $F\colon \cC \to \Sp$ we have a natural equivalence
\[
\Map_{\Fun(\cC,\Sp)}(\map_{\cC}(c, -), F) \simeq \Omega^\infty F(c)
\] 
In our case of interest, the functor $F$ will be $\L$-theory which does not preserve finite limits (unless one inverts 2). This is the ultimate reason why we can only produce functors from connective $\K$-theory to $\L$-theory (and see \cref{thm} which implies that this is the best one can achieve).
\end{Rmk}

\begin{Cor}\label[Cor]{Cor:GW_top}
The functor $\tau_{\geq 0}\Gw_\top\colon \KK \to \Sp_{\geq 0}$ is canonically equivalent to $\k \oplus \k$.
\end{Cor}
\begin{proof}
From the fibre sequence 
\[ \k_{hC_2} \lto \Gw_\top \lto \L \]
and \cref{Cor:KK-invariance}, we deduce that $\Gw_\top$ is a KK-invariant functor. In addition, we now show that the induced functor $\tau_{\geq 0}\Gw_\top \colon \KK \to \Sp_{\geq 0}$ is excisive, i.e.\ sends pushout diagrams to pullback diagrams. For this, we consider again the pullback diagram
\[\begin{tikzcd}
	\tau_{\geq 0}\Gw_\top \ar[r] \ar[d] & \tau_{\geq 0}\L \ar[d] \\
	\tau_{\geq 0}\k^{hC_2} \ar[r] & \tau_{\geq 0}\k^{tC_2}
\end{tikzcd}\]
of functors taking values in connective spectra. We wish to show that, as such $\tau_{\geq 0}\Gw_\top$ is excisive. Note that this means that pullbacks are taken in connective spectra, so being excisive as connective spectrum valued functor is not the same as being excisive when viewed as a spectrum valued functor (via the canonical inclusion of connective spectra in all spectra). We now observe that $\tau_{\geq 0}\k^{hC_2}$ is excisive, so it suffices to show that the fibre of the right vertical map is excisive as well. For this, let us consider a pullback diagram of $C^*$-algebras
\[\begin{tikzcd}
	A \ar[r] \ar[d] & B \ar[d] \\
	A' \ar[r] & B'
\end{tikzcd}\]
whose vertical maps are surjections with a completely positive contractive split (any pullback diagram in $\KK$ is the image of such a diagram e.g.\ by \cite[Theorem 1.4 (2) \& Lemma 2.14]{BEL}). We need to show that the fibre of $\tau_{\geq0} \L \to \tau_{\geq 0}\k^{tC_2}$ sends this square to a pullback diagram of connective spectra. We then consider the commutative square
\begin{equation}\label{diag:excisiveness}
\begin{tikzcd}
	\L(A') \oplus_{\L(A)} \L(B) \ar[r] \ar[d] & \L(B') \ar[d] \\
	\k(A')^{tC_2} \oplus_{k(A)^{tC_2}} \k(B)^{tC_2} \ar[r] & \k(B')^{tC_2} 
\end{tikzcd}
\end{equation}
both of whose horizontal cofibres are given by 
\[ \big[\coker(\k_0(A') \oplus \k_0(B) \to \k_0(B') \big]^{tC_2},\]
see \cite[Theorem 4.2]{LN} for the upper horizontal one. In fact, the proof in loc.\ cit.\ gives that the induced map on horizontal cofibres is an equivalence, therefore diagram \eqref{diag:excisiveness} is a pullback diagram. Consequently, the map on vertical fibres is also an equivalence. This shows that $\mathrm{fib}(\L \to \k^{tC_2})$ is excisive when viewed as a spectrum valued functor, and consequently its connective cover, which agrees with the fibre of $\tau_{\geq 0} \L \to \tau_{\geq 0} \k^{tC_2}$, is excisive when viewed as a connective spectrum valued functor.

We now observe that $\pi_0(\Gw_\top(A))$ is naturally isomorphic to $\pi_0(\k(A) \oplus \k(A))$, induced by sending a tuple $(P_1,P_2)$ of projective modules over $A$ to the hermitian form which is the canonical positive definite form on $P_1$ and the canonical negative definite form on $P_2$, see \cite[Theorem 2.3]{Karoubi}. Since $\k$ is the connective cover of $\K$, which is corepresented by $\R$, and $\Gw_\top$ is product preserving the natural transformation $\pi_0(\k(A) \oplus \k(A)) \to \pi_0(\Gw_\top(A))$ extends to a  transformation $\k\oplus \k \to \Gw_\top(A)$ which induces an isomorphism on $\pi_0$, see  \cref{remark_yoneda}. Since both sides are excisive when viewed as taking values in connective spectra, we deduce that this map induces an isomorphism in $\pi_n$ for all $n\geq 0$: Indeed the just described canonical transformation induces a commutative diagram
\[\begin{tikzcd}
 	\k(\Omega^n A) \oplus \k(\Omega^n A) \ar[r] \ar[d] & \Gw_\top(\Omega^n A) \ar[d] \\
	\Omega^n \k(A) \oplus \Omega^n \k(A) \ar[r] & \Omega^n \Gw_\top(A) 
\end{tikzcd}\]
whose vertical maps are equivalences after taking connective covers (by the established fact that both functors are excisice with values in connective spectra). Now the upper horizontal map induces an isomorphism on $\pi_0$, consequently so does the lower horizontal map. This establishes that $\k(A) \oplus \k(A) \to \Gw_\top(A)$ also induces an isomorphism on $\pi_n$ as claimed.
\end{proof}

\begin{Rmk}\label[Rmk]{Remark:alternative-argument-for-KK-invariance}
One can also give a direct argument for a natural equivalence $\tau_{\geq 0} \Gw_\top(A) \simeq \k(A) \oplus \k(A)$, see \cite[Theorem 2.3]{Karoubi} for the version on homotopy groups. Informally, the map from right to left is obtained as follows: First, one shows that $\tau_{\geq 0} \Gw_\top(A)$ is the group completion of the topological category $\Unimod(A)$ of unimodular hermitian forms over $A$, see e.g.\ \cite[Corollary A.2]{Schlichting}. Then one shows that the functor $\Proj(A) \times \Proj(A) \to \Unimod(A)$, given by sending $(P_1,P_2)$ to $(P_1\oplus P_2, \sigma^\pos\oplus \sigma^\neg)$ is an equivalence of topological categories; here, $\sigma^\pos$ denotes the canonical positive definite form on $P_1$ and $\sigma^\neg$ its negative definite variant. The main statement here is to see that the group of isometries of $(P,\sigma^\pos)$ is homotopy equivalent to the group of isomorphisms of $P$; a shadow of this fact is \cite[Lemma 2.9]{Karoubi}.

This perspective shows that the equivalence in fact holds more generally for $C$-algebras in the sense of \cite[Definition 2.2]{Karoubi}, but we shall not make use of this fact in this paper.

Having this equivalence, one deduces that $\tau_{\geq0} \Gw$ is KK-invariant. From the fibre sequence
\[ \k_{hC_2} \lto \tau_{\geq 0}\Gw_\top \lto \tau_{\geq 0} \L \]
it then follows that $\tau_{\geq 0}\L$, and therefore by periodicity also $\L$, is also KK-invariant. However, this perspective does not immediately give a proof of \cref{ThmC}. We have  decided to deduce the description of $\Gw_\top$ in the way presented rather than showing the equivalence $\tau_{\geq 0}\Gw_\top(A) \simeq \k(A) \oplus \k(A)$ by hand, which might in fact be the more natural thing to do.
\end{Rmk}

\section{Proof of Theorem \ref{ThmA} \& \ref{ThmB}}

In this section, we prove \cref{ThmA} and \cref{ThmB} from the introduction. 
Again, we recall the statements here for convenience. We emphasize at this point that the proofs of Theorem \ref{ThmA} and \ref{ThmB} rely only on the consequence of \cref{ThmC} that L-theory is a KK-invariant functor, not on \cref{ThmC} itself. In particular, Theorems \ref{ThmA} and \ref{ThmB} can also be derived using the argument outlined in \cref{Remark:alternative-argument-for-KK-invariance}. This approach makes no use of the fact that $\L\adjt$ is KK-invariant, which was deduced in \cite{LN} from the fact that $\L$ is KK-invariant on \emph{complex} $C^*$-algebras, which in turn was proven by using that \cref{ThmB} was known previously for complex $C^*$-algebras as indicated in the introduction \cite{Karoubi, Miller, Rosenberg}. 

\begin{Thm}\label[Thm]{Thm:existence+uniqueness-of-tau}
There is a unique lax symmetric monoidal transformation $\tau \colon \k \to \ell$ and the induced maps
\[ \k(A) \otimes_{\ko} \ell(\R) \lto \ell(A) \quad \text{ and } \quad \k(A) \otimes_\ko \L(\R) \lto \L(A) \]
are equivalences for each $C^*$-algebra $A$.
\end{Thm}

\begin{proof}[Proof of \cref{Thm:existence+uniqueness-of-tau}]
Since the canonical map $\ell(A) \otimes_{\ell(\R)} \L(\R) \to \L(A)$ is an equivalence, the second displayed map is obtained from the first by applying the functor $- \otimes_{\ell(\R)} \L(\R)$. It therefore suffices to prove that the first of the two displayed maps is an equivalence\footnote{The statement that the first map is an equivalence is equivalent to the statement that the second map is an equivalence \emph{and} the statement that $\k(A) \otimes_\ko \tau_{<0} \L(\R)$ is coconnected. Using the Whitehead filtration of $\tau_{<0}\L(\R)$ which has graded pieces given by $\Z[4k]$ for $k \leq -1$, and the presentation $\Z = (\ko/\eta)/\beta_\C$ one finds that $\k(A) \otimes_\ko \tau_{<0}\L(\R)$ is indeed coconnected, so the two statements of \cref{Thm:existence+uniqueness-of-tau} are in fact equivalent.}

By the results of the previous section, we know that we may view $\L$ as a functor $\KK \to \Sp$. As such, it is canonically lax symmetric monoidal, because L-theory is lax symmetric monoidal on $C^*$-algebras. In other words, $\L$ is canonically an object of $\Alg( \Fun(\KK,\Sp))$ where algebras are formed with respect to the Day convolution symmetric monoidal structure on $\Fun(\KK,\Sp)$. As such it receives a unique algebra map from the unit, which is given by the functor $\map_\KK(\R,-) \simeq \k$. 
Therefore, as in \cite{LN} there is a unique lax symmetric monoidal transformation $\tau \colon \k \to \L$. 

We now consider the cofibre sequence
\begin{equation} \label{sequence-hyp}
	\k_{hC_2} \stackrel{\hyp}{\lto} \tau_{\geq 0} \Gw_\top \lto \tau_{\geq 0} \L, 
\end{equation}
see \cref{Cor:fibre-sequence-top}, and identify $\tau_{\geq 0}\Gw_\top$ with $\k\oplus \k$ using \cref{Cor:GW_top}. First, we show that the $C_2$-action on $\k$ is trivial: 
To this end we note that $\k$ is corepresented by the tensor unit $\R$ of $\KK$ and that the action is lax symmetric monoidal, so the action is equivalently given by an action on the tensor unit $\R$ in $\KK$. Since the unit in any symmetric monoidal $\infty$-category is the initial commutative algebra object, there is exactly one such action which is thus necessarily the trivial action.
We deduce that
under the equivalence $\tau_{\geq0}\Gw_\top \simeq \k \oplus \k$, the map $\hyp \colon \k_{hC_2} \to \k \oplus \k$ 
can equivalently be described by a $C_2$-equivariant map $r^*(\k) \to r^*(\k\oplus \k)$ where $r\colon BC_2 \to \ast$ is the unique map. 
Therefore, the map $\hyp$ is equivalently described by a map in the category $\Fun(\KK,\Fun(BC_2,\Sp)) \simeq \Fun(BC_2,\Fun(\KK,\Sp))$. The Yoneda Lemma induces the fully faithful inclusion
\[ \Fun(BC_2,\KK^\op) \lto \Fun(BC_2,\Fun(\KK,\Sp)) \]
and the map $\hyp$ is a map between objects in the image. Therefore, the map $\hyp$ in the fibre sequence \eqref{sequence-hyp} is uniquely determined by the associated element  of
\[
\Map(BC_2, \Omega^\infty(\ko \oplus \ko)) 
\]
corresponding to $\hyp(\R)$ evaluated on the element $1 \in \Omega^\infty(\ko)$. 
Similarly, we consider the map 
\[ \k(A) \otimes_\ko \ko_{hC_2} \xrightarrow{\k(A)\otimes_{\ko}\hyp(\R)} \k(A) \otimes_\ko (\ko \oplus \ko) \ . \]
Source and target are canonically equivalent to source and target of our map $\hyp$ and thus this map is also determined by an element in the space $\Map(BC_2, \Omega^\infty(\ko \oplus \ko))$. By construction these two elements in this space agree, since for $A = \R$, the two maps under investigation agree.
So we deduce that the map $\hyp(A)$ identifies with the map $\id_{\k(A)}\otimes_\ko \hyp(\R)$. Therefore, we deduce that 
\[ \ell(A)= \cofib(\hyp(A)) = \k(A) \otimes_\ko \cofib(\hyp(\R)) = \k(A)\otimes_\ko \ell(\R) \]
as claimed. To see that the map is the one we claimed, it suffices to note that the induced map
\[ \k(A) \lto \k(A) \otimes_\ko \ell(\R) \stackrel{\simeq}{\lto} \ell(A) \]
is natural in $A$ and for $A=\R$ agrees with the map $\tau_\R \colon \ko \to \ell(\R)$.
\end{proof}

\begin{Rmk} 
Let us consider the following commutative diagram
\begin{equation}\label{diag:interpretation-of-tau}
\begin{tikzcd}
	\k \ar[r,dashed, "\Delta"] \ar[d] & \k \oplus \k \ar[r,"\ominus"] \ar[d,"\simeq"] & \k \ar[d, dashed, "\hat{\tau}"] \\
	\k_{hC_2} \ar[r,"\hyp"] & \tau_{\geq0}\Gw_\top \ar[r] & \ell
\end{tikzcd}
\end{equation}
where the left vertical map is is the canonical projection map and the middle vertical map is the equivalence of \cref{Remark:alternative-argument-for-KK-invariance}. As a consequence, the (induced) map $\Delta$ is indeed the diagonal map. We obtain a canonical map $\hat{\tau}$ induced on horizontal cofibres. Now we may consider the composite 
\[\k \stackrel{(\id,0)}{\lto} \k \oplus \k \stackrel{\simeq}{\lto} \tau_{\geq0}\Gw_\top \]
considered as a map from the top right term in the above diagram. This map has the following interpretation: It arises by observing that $\k(A)$ can be described as the K-theory of the topological category of positive definite forms on projective $A$-modules. The canonical inclusion to the category of all unimodular forms then induces the map $\k \to \tau_{\geq0}\Gw_\top$ just explained. With this interpretation, one sees that this map is canonically a lax symmetric monoidal transformation. Using that also the map $\Gw_\top \to \L$ is lax symmetric monoidal, see \cite{CDHIV} for a general statement along these lines, we find that the composite 
\[ \k \to \tau_{\geq0}\Gw_\top \to \ell \]
on the one hand agrees with $\hat{\tau}$ (by construction) and is canonically lax symmetric monoidal. By the uniqueness part of \cref{Thm:existence+uniqueness-of-tau}, we deduce that $\hat{\tau} = \tau$.
By expanding out vertical (co)fibres of diagram \eqref{diag:interpretation-of-tau}, we find that $\ell$ is described as the cofibre of a transformation
\[ \widetilde{\k_{hC_2}} \lto \k \]
where the tilde denotes reduced $C_2$-orbits, i.e.\ the cofibre of the projection map $\k \to \k_{hC_2}$. This transformation is, similarly as before, determined by its induced map $\widetilde{\ko_{hC_2}} \to \ko$. A natural guess is that this map is given as follows. We recall that the $C_2$-action on $\ko$ is trivial, so that the above map is equivalently described by a map $BC_2 \to \End(\ko)$, landing in the component of the zero map. We can then consider the canonical map 
\[ BC_2 \lto \gl_1(\ko) \subseteq \End(\ko) \]
which lands in the component of the identity, and shift it to the component of the zero map using the additive structure on $\End(\ko)$. We note that precomposing this map with the canonical map $B\Z \to BC_2$, we obtain a map $\Sigma \ko \to \ko$ which is given by the multiplication by $\eta$. This would be compatible with the discussion at the end of this section, but we refrain from attempting to prove that the map $\widetilde{\ko_{hC_2}} \to \ko$ is indeed given by this construction.
\end{Rmk}

Next, we aim to prove \cref{ThmB} from the introduction. Following standard topological notation we write $\KSp = \K(\H)$ for the topological $\K$-theory spectrum of the quaternions and denote by $\ksp$ its connective cover.  As a further preparation we denote by $\tilde{\eta} \colon \Sigma\ko/2 \to \ko$ a $\ko$-linear extension of the $\eta$-multiplication map $\eta\colon \Sigma \ko \to \ko$ to $\Sigma \ko/2$. Such an extension exists as $2\eta =0$, but is not unique.\footnote{Up to homotopy, there are two extensions since $\pi_2(\ko) = \Z/2$.} Regardless which extension is chosen, we have the following symplectic analogue of Wood's theorem -- recall that Wood's theorem states that $\cofib(\eta) = \ku$ or equivalently that 
the cofibre of the periodic version $\eta: \Sigma \KO \to \KO$ is $\KU$. For a proof of Wood’s theorem see \cite[Theorem 3.2]{AkhilWood}, but note
that the argument was cut from the published version \cite{AkhilnoWood}.
\begin{Lemma}\label[Lemma]{Lemma:KSp-presentation}
There is an equivalence $\cofib(\tilde{\eta}) \simeq \ksp$.
\end{Lemma}
\begin{proof}
Since $\KSp \simeq \Sigma^4 \KO$, we may equivalently show that there is a fibre sequence
\[ \Sigma \KO/2 \stackrel{\tilde{\eta}}{\lto} \KO \lto \Sigma^4 \KO.\]
We will deduce this from the periodic version of Wood's theorem. To this end, consider the commutative diagram
\[\begin{tikzcd}
	\Sigma \KO \ar[r] \ar[d, "\cdot 2"] & 0 \ar[d]\ar[r] & \Sigma^2 \KO\ar[d,dashed] \\
	\Sigma \KO \ar[r, "\eta"] \ar[d]\ar[d]&  \KO \ar[r, "c"]\ar[d,"\mathrm{id}"] & \KU\ar[d, dashed] \\
	\Sigma \KO/2 \ar[r, "\tilde{\eta}"] & \KO \ar[r] & C
\end{tikzcd}\]
where $C = \cofib(\tilde{\eta})$. The upper left square is filled by the choosen nullhomotopy of $2\eta$ and the map $c\colon \KO \to \KU$ is the complexification map that sends $1$ to $1$.
The vertical and horizontal sequences are all fibre sequences, the middle horizontal one by Wood's theorem, the rest by definition. The right vertical sequence is obtained from the left part of the diagram by forming horizontal cofibres. In particular we see that $C$ is the cofibre of the map $\Sigma^2\KO \dashrightarrow\KU$. By construction, this map is $\KO$-linear, since all maps and homotopies in the diagram are $\KO$-linear. Thus the map is determined by its value on the generator $1 \in \pi_2(\Sigma^2\KO) = \pi_0(\KO)$. By looking at the long exact sequences associated to the horizontal fibre sequences we see that this generator is sent to a generator in $\pi_2(\KU) \cong \mathbb{Z}$. Thus (after postcomposing with the invertible and $\KO$-linear map given by multiplication with this generator) the map $\Sigma^2\KO \dashrightarrow \KU$ is equivalent to $\Sigma^2$ of the complexification map $c\colon \KO \to \KU$. By another application of Wood's theorem, we deduce that the cofibre $C$ is given by $\Sigma^4 \KO$ as needed.
\end{proof}

\begin{Cor}\label[Cor]{Cor:3-truncation-of-L}
There is a $\ko$-linear map $\ksp \to \ell(\R)$ which induces an isomorphism on $\pi_0$ and consequently an equivalence $\tau_{\leq 3} \ksp \simeq \tau_{\leq 3} \ell(\R)$ on Postnikov 3-truncations.
\end{Cor}
\begin{proof}
Since the canonical map $\ko \to \ksp$ induced by the map $\R \to \H$ is a $\pi_0$ isomorphism, any extension of $\tau\colon \ko \to \ell(\R) $ along $\ko \to \ksp$ is also a $\pi_0$ isomorphism.
Therefore, by \cref{Lemma:KSp-presentation}, it suffices to show that the composite 
\[ \Sigma \ko/2 \stackrel{\tilde{\eta}}{\lto} \ko \stackrel{\tau}{\lto} \ell(\R) \]
is $\ko$-linearly null-homotopic. But we have 
\[\Map_\ko(\Sigma \ko/2,\ell(\R)) \simeq \Map_\ko(\Sigma \ko,\Omega \ell(\R)/2) \simeq \Omega^{\infty + 2} \ell(\R)/2 \]
which is connected. 
The final equivalence on 3-truncations follows because both spectra have $\pi_1 = \pi_2 = \pi_3 = 0$.
\end{proof}

We are now ready to prove \cref{ThmB} from the introduction, which we state here in a form better suited for describing the comparison map $\tau$ on homotopy groups in all non-negative degrees, see \cref{Rem:induced-map}. We denote by $\K_*(A)[\eta]$ the $\eta$-torsion of $\K_*(A)$, that is, the kernel of the map $\K_*(A) \to \K_{*+1}(A)$ given by multiplication by $\eta$. The cokernel of this map is denoted by  $\K_{*+1}(A)/\eta$.

\begin{Thm}\label[Thm]{Thm:natural-ThmB}
Let $A$ be a $C^*$-algebra. For all $n \in \Z$, there are canonical and natural isomorphisms
\begin{enumerate}
\item $\L_{4n}(A) \cong \K_{8n}(A)$, 
\item $\L_{4n+1}(A) \cong  \K_{8n+1}(A) / \eta$,
\item $\L_{4n+2}(A) \cong \K_{8n+6}(A)[\eta]$, and
\item $\L_{4n+3}(A) \cong \K_{8n+7}(A)$.
\end{enumerate}
\end{Thm}
\begin{proof}
First, we note that it suffices to prove the theorem for $n=0$, as the L-groups are naturally 4-periodic and the K-groups are naturally 8-periodic. Moreover, we note that the periodicity generators $b \in \L_4(\R)$ and $\beta_\R \in \K_8(\R)$ are canonical (not only up to sign), for instance because they are determined by squares in $\L_4(\C)$ and $\K_8(\C)$ respectively.\footnote{Indeed, there are ring maps $\L(\R) \to \L(\C)$ and $\ko \to \ku$, sending $b$ to $b_\C^2$ and $\beta_\R$ to $\beta_\C^4$, respectively, for any choice of generators $b_\C \in \L_2(\C)$ and $\beta_\C \in \K_2(\C)$.} 
Using the presentation $\ell(A) \simeq \k(A) \otimes_\ko \ell(\R)$ obtained in \cref{ThmA} and \cref{Cor:3-truncation-of-L}
we deduce that the map $\ksp \to \ell(\R)$ induces the equivalence 
\[ \tau_{\leq 3}\big(\k(A) \otimes_\ko \ksp\big) \stackrel{\simeq}{\lto} \tau_{\leq 3}\ell(A).\] 
We now utilise that $\k(A) = \tau_{\geq 0}\K(A)$ is the connective cover of a $\KO$-module and proceed with the following general observation. We let $M$ be a $\KO$-module and are then interested in the low degree homotopy of the $\ko$-module 
\[ (\tau_{\geq 0}M) \otimes_\ko \ksp.\]
From the fibre sequence $\Sigma\ko/2 \to \ko \to \ksp$ obtained in \cref{Lemma:KSp-presentation}, we deduce that 
\begin{enumerate}
\item $\pi_0(\tau_{\geq 0}M \otimes_\ko \ksp) \cong \pi_0(M)$, and
\item $\pi_1(\tau_{\geq 0}M \otimes_\ko \ksp) \cong \coker(\pi_0(M) \stackrel{\eta}{\to} \pi_1(M))$.
\end{enumerate}
To calculate $\pi_2$ and $\pi_3$, we consider the following diagram of horizontal and vertical fibre sequences
\[\begin{tikzcd}
	\Sigma (\tau_{\geq 0}M)/2 \ar[r] \ar[d] & \tau_{\geq 0}M \ar[r] \ar[d] & \tau_{\geq 0}M \otimes_\ko \ksp \ar[d] \\
	\Sigma M/2 \ar[r] \ar[d] & M \ar[r] \ar[d] & M \otimes_\KO \KSp \ar[d] \\
	\Sigma (\tau_{<0}M)/2 \ar[r] & \tau_{<0}M \ar[r] & C
\end{tikzcd}\]
from which we deduce that $\pi_i(C) = 0$ for $i\geq 3$ and that 
\[ \pi_2(C) \cong \pi_1(\Sigma (\tau_{<0}M)/2) \cong \pi_{-1}(M)[2].\]
In addition, we note that $\KSp \simeq \Sigma^4 \KO$. We therefore have a fibre sequence
\[ \tau_{\geq 0}M \otimes_\ko \ksp \lto \Sigma^4 M \lto C \]
whose long exact sequence on homotopy groups reveals that $\pi_3(\tau_{\geq 0}M \otimes_\ko \ksp) \cong \pi_3(\Sigma^4M) \cong \pi_{-1}M$ and that there is an exact sequence
\[ 0 \lto \pi_2(\tau_{\geq 0}M \otimes_\ko \ksp) \lto \pi_{-2}(M) \lto \pi_{-1}(M).\]
Here, we have used that $\pi_{2}(C) \subseteq \pi_{-1}(M)$. The latter map in this exact sequence is a natural transformation of functors $\pi_{-2} \to \pi_{-1}$ on $\KO$-modules, and is therefore either trivial or the $\eta$-multiplication. We claim that it is the $\eta$-multiplication, which then shows the theorem. 

The claim is equivalent to the statement that the map
\[ \pi_2(\tau_{\geq 0}M \otimes_\ko \ksp) \lto \pi_{-2}(M) \] 
appearing above is in general not an isomorphism. Therefore, it suffices to find an example of a $\KO$-module $M$ where $\pi_{-2}(M) \neq 0$ but $\pi_2(\tau_{\geq 0}M \otimes_\ko \ksp) = 0$. First, we note that there is an isomorphism
\[ \pi_2(\tau_{\geq 0}M \otimes_\ko \ksp) \simeq \pi_2(\tau_{[0,2]}M \otimes_\ko \Z) \]
Choosing $M = \KO[-3]$, we find $\pi_{[0,2]}M = \Z[1]$, so that 
\[  \pi_2(\tau_{[0,2]}M \otimes_\ko \Z) \cong \pi_1(\Z \otimes_\ko \Z) = 0,\]
in fact, since $\Z \simeq (\ko/\eta)/{\beta_\C}$, we have $\Z \otimes_\ko \Z \simeq (\Z/\eta)/\beta_\C \simeq \Z \oplus \Sigma^2 \Z \oplus \Sigma^3\Z \oplus \Sigma^5\Z$ since $\eta$ and $\beta_\C$ are zero on $\Z$.
However, $\pi_{-2}(M)= \pi_1(\KO) \neq 0$, so the claim is shown.
\end{proof}

We end this section with the following perspective on the map $\k(A) \otimes_\ko \ksp \to \ell(A)$ which was used in the proof of \cref{Thm:natural-ThmB}. Namely, as a consequence of \cref{Thm:existence+uniqueness-of-tau} and \cref{Lemma:KSp-presentation}, there is a commutative diagram
\[\begin{tikzcd}
	\Sigma \k(A)/2 \ar[r] \ar[d] & \k(A) \ar[r] \ar[d, equal] & \k(A) \otimes_\ko \ksp \ar[d] \\
	\widetilde{\k(A)_{hC_2}} \ar[r] & \k(A) \ar[r] & \ell(A)
\end{tikzcd}\]
and the fact that the right vertical map induces an equivalence of $3$-truncations can be used to show that the cofibre of the left most vertical map is 3-connective with $\pi_3$ given by $\k_0(A)/2$. Since this map is obtained from the map $\Sigma\SS/2 \to \widetilde{\SS_{hC_2}}$ upon applying the functor $- \otimes \k(A)$, this result also follows from the following lemma.

\begin{Lemma}\label[Lemma]{Lemma:relation}
There is a map $\Sigma \SS/2 \to \Sigma^\infty BC_2$ whose cofibre is 3-connective with $\pi_3$ isomorphic to $\Z/2$.
\end{Lemma}
\begin{proof}
First, we recall the low dimensional homotopy groups of $\SS/2$ and $\Sigma^\infty BC_2$: We have that $\pi_0(\SS/2) = \Z/2$, $\pi_1(\SS/2) = \Z/2$ and $\pi_2(\SS/2) = \Z/4$. In addition, the $\eta$-multiplications 
\[ \pi_0(\SS/2) \lto \pi_1(\SS/2) \lto \pi_2(\SS/2) \]
are injective, as follows from comparing with $\SS$ along the canonical map $\SS \to \SS/2$. Now, according to \cite{Liulevicius}, we have $\pi_1(\Sigma^\infty BC_2) = \Z/2$, $\pi_2(\Sigma^\infty BC_2) = \Z/2$ and $\pi_3(\Sigma^\infty BC_2) = \Z/8$. The Atiyah--Hirzebruch spectral sequence then shows that the map $\Sigma \SS = \Sigma^\infty B\Z \to \Sigma^\infty BC_2$ induces the projection on $\pi_1$ an isomorphism on $\pi_2$ and an injection on $\pi_3$. In particular, this map descends to a map $\Sigma \SS/2 \to \Sigma^\infty BC_2$ and the induced map then induces an isomorphism on $\pi_0$ and $\pi_1$. On $\pi_3$, the composite $\Sigma \SS \to \Sigma\SS/2 \to \Sigma^\infty BC_2$ identifies with 
\[ \Z/2 \lto \Z/4 \lto \Z/8 \]
where the composite is the non-trivial map. It follows that $\Z/4 \to \Z/8$ must be injective as claimed.
This calculation also shows that the cofibre of the map $\Sigma \SS/2 \to \Sigma^\infty BC_2$ has $\pi_3$ isomorphic to $\coker(\Z/4 \subseteq \Z/8) \cong \Z/2$ as claimed.
\end{proof}

\section{Algebraic structure of $\L_*(-)$}\label{sec:5}

In this section we will describe the algebraic structure on the $\L$-theory groups under the isomorphisms obtained in \cref{Thm:natural-ThmB} and compare our results to previously known results. We will freely use the isomorphisms of \cref{Thm:natural-ThmB} which identifies all $\L$-groups. 

Recall that the homotopy groups $\KO_* = \K_*(\R)$ are $8$-periodic with the (invertible) real Bott element $\beta_\R$ in degree 8. We fix the generator $x \in \K_4(\R) \cong \Z$ whose complexification is $2\beta_\C^2$ and recall the relations $x^2 = 4 \beta_\R$ and $\eta x = 0$.

\begin{Prop}\label[Prop]{Rem:induced-map}
For a $C^*$-algebra $A$, the map $\tau_A: \k(A) \to \ell(A)$ induces the following maps on homotopy groups $\pi_n$ for $n\geq 0$:
\begin{align*}
 (2x)^n\colon \quad & \K_{4n}(A) \lto \K_{8n}(A) \cong \L_{4n}(A) \\
(2x)^n\colon \quad & \K_{4n+1}(A) \lto \K_{8n+1}(A)/ \eta \cong \L_{4n+1}(A) \\
(2x)^n \cdot x \colon \quad & \K_{4n+2}(A) \lto \K_{8n+6}(A)[\eta] \cong \L_{4n+2}(A) \\
(2x)^n\cdot x\colon \quad & \K_{4n+3}(A) \lto \K_{8n+7}(A) \cong \L_{4n+3}(A) \ .
\end{align*}
\end{Prop}
\begin{Rmk}\label[Rmk]{Rmk:complex-case}
In particular, in degree $8n$ and $8n+1$ the map under investigation is given by multiplication by $16^n$, up to Bott periodicity isomorphisms. Since the map $\K_4(\R) \to \K_4(\C)$ sends $x$ to $2\beta_\C^2$, \cref{Rem:induced-map} also shows that for complex $C^*$-algebras, the map induces multiplication by $2^n$ on $\pi_{2n}$ and $\pi_{2n+1}$. This was previously obtained in \cite[Theorem 4.1]{LN} and we shall make use of this fact below.
\end{Rmk}
\begin{proof}[Proof of \cref{Rem:induced-map}]
We note that the assignment $A \mapsto \K_n(A)$ viewed as a functor $\KK \to \Ab$ is corepresented by an $n$-fold shift (i.e.\ suspension) of $\R$, which we denote by $\R[n]$. Therefore, the Yoneda lemma for product preserving functors $\KK \to \Ab$ implies that natural transformations $\K_n \to \L_n$ are in 1-1 correspondence to classes in $\L_n(\R[n])$.
By \cref{ThmB}, this group is isomorphic to $\K_0(\R) = \Z$ and $\K_4(\R) = \Z\{x\}$ when $n\equiv 0,1,6,7 \mod 8$ and $n \equiv 2,3,4,5\mod 8$, respectively. We deduce that maps $\K_n(A) \to \L_n(A)$ have to be given by multiplication with a multiple of $x$ or a multiple of $1$ (under the respective identifications depending on $n$ described above). From the case of complex $C^*$-algebras as discussed in \cref{Rmk:complex-case}, we immediately deduce the precise form of the multiple: we simply note that for a complex $C^*$-algebra $A$ the element $x$ acts as $2\beta_\C^2$. Therefore, $(2x)^n$ acts as $2^{2n}$, and $(2x)^n\cdot x$ acts as $2^{2n+1}$. Therefore, the formulas described in the statement of \cref{Rem:induced-map} are correct for complex $C^*$-algebras (this is the content of \cref{Rmk:complex-case}), and hence by the above analysis in general.
\end{proof}

Next, we want to explain how the lax symmetric monoidal structure on $\L_*(-)$ is described in terms of the lax symmetric monoidal structure on $\K_*(-)$ under the isomorphisms provided by Theorem B. To state the result we have to describe the maps 
\[ \L_i(A) \otimes \L_j(B) \lto \L_{i+j}(A \otimes B) \] 
for $i,j = 0,1,2,3 \mod 4$ as everything is multiplicatively $4$-periodic. By graded symmetry, it is enough to do this for $0 \leq i \leq j \leq 3$.  We denote the lax symmetric monoidal structure of $\K$-theory  as
\[ \K_i(A) \otimes \K_j(B) \lto \K_{i+j}(A \otimes B) \qquad (a,b) \mapsto a * b \] 
and the induced $\KO_* = \K_*(\R)$-module structure on $\K_*(A)$ by the multiplication sign.

\begin{Prop}\label[Prop]{Rem:multiplication}
Under the isomorphisms of \cref{Thm:natural-ThmB} the exterior multiplication maps on the L-groups are maps of the following kind. 
\begin{enumerate}
\item $\K_{8n}(A) \otimes \K_{8m}(B)   \lto \K_{8n+8m}(A \otimes B)$
\item $\K_{8n}(A) \otimes \K_{8m+1}(B) / \eta \lto \K_{8n+8m + 1}(A \otimes B)/ \eta$
\item $\K_{8n}(A) \otimes \K_{8m+6}(B)[\eta] \lto \K_{8n+8m+ 6}(A \otimes B)[\eta]$
\item $\K_{8n}(A) \otimes \K_{8m+7}(B) \lto \K_{8n+8m+7}(A \otimes B)$
\item $\K_{8n+1}(A)/\eta \otimes \K_{8m+1}(B) / \eta \lto \K_{8n+8m+6}(A \otimes B)[\eta]$
\item $\K_{8n+1}(A)/\eta \otimes \K_{8m+6}(B)[\eta] \lto \K_{8n+8m+7}(A \otimes B)$
\item $\K_{8n+1}(A)/\eta \otimes \K_{8m+7}(B) \lto \K_{8n+8m+8}(A \otimes B)$
\item $\K_{8n+6}(A)[\eta] \otimes \K_{8m+6}(A)[\eta] \lto \K_{8n+8m+8}(A \otimes B)$
\item $\K_{8n+6}(A)[\eta] \otimes \K_{8m+7}(A) \lto \K_{8n+8m+9}(A \otimes B)/\eta$
\item $\K_{8n+7}(A)\otimes \K_{8m+7}(A) \lto \K_{8n+8m+14}(A \otimes B)[\eta]$
\end{enumerate}
For $a$ belonging to the left tensor factor and $b$ belonging to the right tensor factor, these maps are given by the following formulas:
\begin{figure}[H]
\begin{tabular}{|l|l|l|l|l|}
\hline
                     & $\K_{8m}(B)$ & $\K_{8m+1}(B)/\eta$ & $\K_{8m+6}(B)[\eta]$ & $\K_{8m+7}(B)$ \\ \hline
$\K_{8n}(A)$         &    $a*b$         &         $a*b$           &        $a*b$              &       $a*b$         \\ \hline
$\K_{8n+1}(A)/\eta$  &             &       $x\cdot (a*b)$  &  $a*b$          &       $2(a*b)$                 \\ \hline
$\K_{8n+6}(A)[\eta]$ &            &                     &            $\tfrac{x}{2\beta_\R}(a*b)$          &   $\tfrac{x}{2\beta_\R}(a*b)$             \\ \hline
$\K_{8n+7}(A)$       &            &                     &                      &   $2(a*b)$             \\ \hline
\end{tabular}
\end{figure}

Here, the abusive term $\frac{x} {2\beta_{\R}} (a * b)$ denotes an element depending naturally on $a$ and $b$ and whose multiplication with $2$ is given by $\frac{x} {\beta_{\R}} (a * b)$. Part of the statement is the claim that  there is a unique such element. 
\end{Prop}
\begin{proof}

We take a step back again and recall that the exterior multiplication maps on L-theory are natural transformations 
\[
\L_i \otimes \L_j \to \L_{i + j} \ .
\]
If $i$ and $j$ are $0$ or $3$ modulo $4$, then the $\L$-groups are isomorphic to $\K$-groups and thus the source $\L_i \otimes \L_j$ is corepresentable by shifts of $\R$ (on the category $h\KK \otimes h\KK$ whose objects are pairs of $C^*$-algebras and whose hom abelian groups are the tensor products of the hom abelian groups in $h\KK$). Consequently, the exterior multiplication $\L_i \otimes \L_j \to \L_{i+j}$ is given by an element in $\L_{i+j}(\R[m])$ for appropriate $m$. This group is isomorphic to $\K_0(\R) = \Z$ and $\K_0(\R)[\eta] =  2\Z$ (depending on the precise values of $i$ and $j$) so that in these cases the multiplication has to be given by a multiple of $a * b$ and $2(a*b)$, respectively. Using that the map of  \cref{Rem:induced-map} is to be compatible with external products, we immediately get the desired multiples. This proves cases $(1)$, $(4)$, and $(10)$. 

By \cref{ThmB} and the remark following \cref{ThmB} in the introduction, we have natural surjections $\K_1(A) \to \L_1(A)$ and $\K_0(A_\C) \to \L_2(A)$.
The functor $A \mapsto \K_0(A_\C)$ is corepresentable by $\C$ and the functor $A \mapsto \K_1(A)$ by $\R[1]$. We deduce that for any values of $i$ and $j$ we have a surjection $F_i(A) \otimes F_j(A) \twoheadrightarrow \L_i(A) \otimes \L_j(A)$ where $F_i$ and $F_j$ are corepresentable. Any natural transformation with source $\L_i \otimes \L_j$ is then uniquely determined by its restriction to $F_i \otimes F_j$. Computing natural transformations $F_i \otimes F_j \to \L_{i + j}$ the resulting groups are given by 
\[
\K_0(\R) , \K_0(\C), \K_4(\R), \K_0(\C\otimes_\R \C).
\]
Using again that the comparison map $\tau$ is compatible with external products and \cref{Rem:induced-map}, we obtain cases (2), (3), (5), (6), and (7). It remains to treat case (8) and (9). We shall argue case (8) and leave the details of case (9) to the reader. We consider the following diagram of natural transformations
\[\begin{tikzcd}
	\K_2(A_\C) \otimes \K_2(B_\C) \ar[r,"u\otimes u"] \ar[d,"x\otimes x"] & \K_2(A) \otimes \K_2(B) \ar[r] \ar[d,"x\otimes x"] & \K_4(A\otimes B) \ar[d,"2x"] \\
	\K_6(A_\C) \otimes \K_6(B_\C) \ar[r, two heads,"u\otimes u"] & \K_6(A)[\eta]\otimes \K_6(B)[\eta] \ar[r,"m"] & \K_8(A\otimes B)
\end{tikzcd}\]
the left diagram commutes because the forgetful map $u\colon \K(A_\C) \to \K(A)$ is $\KO$-linear. The right diagram commutes because the vertical maps are induced by $\tau$, see \cref{Rem:induced-map}, and $\tau$ is compatible with exterior multiplications.
The lower right horizontal map $m$ is the one we wish to describe as $m(a,b) = \tfrac{x}{2\beta_\R}(a*b)$. Since the lower left horizontal arrow is surjective, it suffices to show that the equality holds after precomposing with this surjective map. Doing this, both terms are natural transformations which, by corepresentability of the source, are given by elements of $\K_0(\C\otimes \C)$.
Since this group is torsion free, we may equivalently show that
\[16m(u(a),u(b)) = \tfrac{8x}{\beta_\R}(u(a)*u(b))\]
where $a \in \K_6(A_\C)$ and $b \in \K_6(B_\C)$.
Using the above commutative diagram, and the fact that the on the K-theory of complex $C^*$-algebras, $x$ acts via $2\beta_\C^2$, we see that 
\[ 4m(u(a),u(b)) = 2x \cdot (u(\beta_\C^{-2} a) * u(\beta_\C^{-2}b))\]
so it suffices to show that 
\[ 4\beta_\R \cdot (u(\beta_\C^{-2}a) * u(\beta^{-2}_\C b)) = 4(u(a) * u(b)).\]
This follows from the facts that $4\beta_\R = x^2$, $-*-$ is $\KO$-bilinear, $u$ is $\KO$-linear, and that $xa = 2\beta_\C^2 a$ as already used earlier.
Case (9) can be shown by a similar argument.
\end{proof}

\begin{Rmk}
In this remark, we collect what was previously known about the L-groups of $C^*$-algebras. 
\begin{enumerate}
\item There is a canonical signature-type isomorphism $\L_0(A) \to \K_0(A)$, see e.g.\ \cite[Theorem 1.6]{Rosenberg}. 
\item There are canonical isomorphisms $\K_n(A)\adjt \cong \L_n(A)\adjt$, see \cite[Theorem 1.11]{Rosenberg} and \cite{LN} for the more general statement that there is a canonical equivalence $\K\adjt \simeq \L\adjt$ of spectrum valued functors.
\item For a unital real $C^*$-algebra $A$, there is a canonical surjection $\K_1(A) \to \L_1^{h}(A)$ whose kernel is generated by the image of $\K_1(\R) \to \K_1(A)$, see \cite[Theorem 1.9]{Rosenberg}. Here, $\L^h$ refers to \emph{free} L-theory. We give a new proof of this presentation of $\L_1^h(A)$ in \cref{Prop:free-L-theory-degree1+3} below. 
With the arguments used there, an observation about the Rothenberg sequence for $C^*$-algebras \cite[Remark 1.7]{Rosenberg}, and some additional work, one can in fact conversely recover an isomorphism $\L_1(A) \cong \K_1(A)/\eta$ for any $C^*$-algebra $A$.
\end{enumerate}
To the best of our knowledge, no conjectural relation between $\L_n(A)$ and $\K_n(A)$ has been made for $n \geq 2$ without inverting 2. We also note that, by construction, the isomorphism in (1) is the inverse of the canonical isomorphism induced by the map $\tau$ of \cref{ThmA}. 
\end{Rmk}
We will now also comment how our results imply statements about the higher free L-groups of unital $C^*$-algebras and in particular reprove part (3) above in \cref{Prop:free-L-theory-degree1+3}.
First, we describe the free L-theory of a unital $C^*$-algebra as follows. We recall that $SA=C_0((0,1);A)$ denotes the $C^*$-algebraic suspension of the algebra $A$ and note that $S$ descends to the loop functor on the stable $\infty$-category $\KK$.

\begin{Prop}\label[Prop]{prop:free-L}
Let $A$ be a unital $C^*$-algebra. There is a canonical fibre sequence
\[ \Sigma \L(SA) \lto \L^h(A) \lto C(A)^{tC_2} \]
where $C(A) = \ker(\K_0(A) \to \widetilde{\K}_0(A)) = \mathrm{Im}(\K_0(\R) \to \K_0(A))$ is a cyclic group.
\end{Prop}
\begin{proof}
By \cite[Proposition 4.6]{LN} and the Rothenberg sequence for $\L^h(-)$ and $\L(-)$ \cite[\S 9]{Ranicki}, we have a commutative diagram of fibre sequences (the left vertical dashed map is the one induced from the right solid square)
\begin{equation}\label{diag:fibre-sequences} 
\begin{tikzcd}
	\Sigma \L(SA) \ar[r] \ar[d, dashed] & \L(A) \ar[d,equal] \ar[r] & \K_0(A)^{tC_2} \ar[d] \\
	\L^h(A) \ar[r] & \L(A) \ar[r] & \widetilde{\K}_0(A)^{tC_2} 
\end{tikzcd}
\end{equation}
from which the proposition follows immediately.
\end{proof}

The following is an amusing consequence.
\begin{Cor}
Suppose $A$ is a unital $C^*$-algebra in which the element $[A] \in \K_0(A)$ has odd order\footnote{e.g. $A = \mathscr{O}_{2n}^\R$.}. Then the map $\Sigma \L(SA) \to \L^h(A)$ is an equivalence.
\end{Cor}
\begin{proof}
The element $[A] \in \K_0(A)$ generates the kernel of the map $\K_0(A) \to \widetilde{\K}_0(A)$. Therefore, under the assumptions of the corollary, $C(A)$ is a finite group of odd order, so its $C_2$-Tate cohomology vanishes.
\end{proof}

We then investigate the long exact sequence associated to the fibre sequence of \cref{prop:free-L}. To do so, we first analyse the top horizontal fibre sequence in diagram \eqref{diag:fibre-sequences} and recall that, since the $C_2$-action on $\K_0(A)$ is trivial\footnote{any finitely generated projective $A$-module $P$ admits a positive definite unimodular form, giving an isomorphism from $P$ to $P^\vee$. See also the proof of \cref{Thm:existence+uniqueness-of-tau} for the triviality of the $C_2$-action on the spectrum $\k(A)$.}, we have $\hat{H}^\mathrm{ev}(C_2;\K_0(A)) \cong \K_0(A)/2$ and $\hat{H}^\mathrm{odd}(C_2;\K_0(A)) \cong \K_0(A)[2]$. 
\begin{Prop}\label[Prop]{prop:boundary-maps}
Under the isomorphisms provided by \cref{ThmB}, the natural maps $\L_n(A) \to \pi_n(\K_0(A)^{tC_2})$ appearing in the long exact sequence associated to the top horizontal fibre sequence of diagram \eqref{diag:fibre-sequences} are the following ones:
\begin{enumerate}
\item the projection $\K_0(A) \to \K_0(A)/2$ for $n \equiv 0 \mod 4$,
\item the trivial map $\K_1(A)/\eta \to \K_0(A)[2]$ for $n\equiv 1 \mod 4$,
\item the unique, non-trivial natural map $\K_6(A)[\eta] \to \K_0(A)/2$ for $n \equiv 2 \mod 4$\footnote{The assertion is that there is exactly one such natural map. An explicit description can be given as follows: lift an element in $K_6(A)[\eta]$ to an element in $K_6(A_\C)$ along the `forgetful' map $K_*(A_\C) \to K_*(A)$. Then multiply the lift with $\beta^{-3}$ to obtain an element in $K_0(A_\C)$ and apply the forgetful map followed by the mod 2 reduction.
}\label{foot}
, and
\item the multiplication by $\eta$ map $\K_7(A) \to \K_0(A)[2]$ for $n \equiv 3 \mod 4$.
\end{enumerate}
\end{Prop}
\begin{proof}
First, we note that all maps appearing are natural in $A$. Next, we recall that the map under consideration is the composite of the natural transformations $\L \to \K_\alg(-)^{tC_2} \to \K_0(-)^{tC_2}$, both of which are canonically lax symmetric monoidal transformations. We deduce that the map under consideration is $4$-periodic (since everything is a module over $\L(\R)$), hence it suffices to treat the cases $n=0,1,2,3$.
The case $n=0$ follows from a direct inspection. The case $n=1$ is obtained by considering the natural maps 
\[ \K_1(A) \lto \K_1(A)/\eta \lto \K_0(A)[2] \subseteq \K_0(A) \]
and observing that any such natural map is given by multiplication by an element of $\K_{-1}(\R) = 0$. Since the first map above is surjective and the last map is injective, the middle map is trivial as claimed. The case $n=2$ is obtained by noting that the composite 
\[ \K_6(A_\C) \lto \K_6(A)[\eta] \lto \K_0(A)/2 \]
is again natural and the first map is surjective by the generalised Wood sequence discussed in the remark following \cref{ThmB}. Furthermore, as before, the source is, as a functor in $A$, corepresented by $\C$. Therefore, natural such maps are given by an element in $\K_0(\C)/2 = \Z/2$. It then suffices to show that the map under investigation is not trivial. This follows from the case of complex $C^*$-algebras: The 2-periodicity of L-theory for complex $C^*$-algebras indeed shows that this map identifies with the map for $n=0$ which is non-trivial by the first part. 
This construction also shows that the description given in footnote 5 is correct: the map $\K_6(A_\C) \to \K_0(A)/2$ given by the non-trivial element $\K_0(\C)/2 = \Z/2$ factors as $\K_6(A_\C) \to \K_0(A_\C) \to \K_0(A) \to \K_0(A)/2$ since the element in $\K_0(\C)/2$ lifts through the induced maps $ \K_0(\C_\C) \to \K_0(\C) \to \K_0(\C)/2$ in the prescribed manner.

Finally, the case $n = 3$ must, by the same reasoning as earlier, be given by multiplication with a 2-torsion element of $\K_1(\R)=\Z/2\{\eta\}$. It then suffices to know that this map is non-trivial, which follows from considering the algebra $A= S\R$.
\end{proof}

\begin{Cor}\label[Cor]{Prop:free-L-theory-degree1+3}
Let $A$ be a unital $C^*$-algebra. Then 
\begin{enumerate}
\item there is a canonical isomorphism $\L_1^h(A) \cong \K_1(A)/\langle \eta \rangle$, and
\item there is a canonical isomorphism $\L_3^h(A) \cong \K_7(A) \times_{\K_8(A)} C(A)$, where the maps appearing in the pullback are given by the $\eta$ multiplication $\K_7(A) \to \K_8(A)$ and the canonical map $C(A) \to \K_0(A) \cong \K_8(A)$.
\end{enumerate}

\end{Cor}
\begin{proof}
To prove part (1), we consider the following diagram 
\[\begin{tikzcd}
	& C(A)/2 \ar[r] \ar[d] & \L_0(SA) \ar[r] \ar[d,equal] & \L_1^h(A) \ar[r,"0"] \ar[d] & C(A)[2] \ar[d, hookrightarrow] \\
	\L_2(A) \ar[r] & \K_0(A)/2 \ar[r,"\eta"] & \L_0(SA) \ar[r] & \L_1(A) \ar[r] & \K_0(A)[2] \ar[r, hookrightarrow] & \L_{-1}(SA) 
\end{tikzcd}\]
and note first that $\L_0(SA) \cong \K_1(A)$ by \cref{ThmB}.
Then we observe that the right most vertical arrow is injective, simply because $C(A) \to \K_0(A)$ is. Furthermore, the right most bottom horizontal arrow is injective, see \cite[Remark 1.10]{Rosenberg} and the argument used in the proof of \cite[Proposition 4.6]{LN} via diagram (1) therein. It follows that the right most top horizontal arrow is trivial, so that $\L_1^h(A)$ is a natural quotient of $\K_1(A)$. To see which precise quotient it is, we observe again by naturality, that the second to left most bottom horizontal arrow is given by multiplication by $\eta$: It is either that or trivial, and the case of $A= \R$ shows that the map is non-trivial since $\L_2(\R) = 0$.

To prove part (2), we consider the same exact sequences shifted in the appropriate degrees:
\[\begin{tikzcd}
	& & \L_2(SA) \ar[d, equal] \ar[r] & \L_3^h(A) \ar[r] \ar[d] & C(A)[2] \subseteq C(A) \ar[d, hookrightarrow] \\
	\L_4(A) \ar[r, two heads] & \K_0(A)/2 \ar[r,"0"] & \L_2(SA) \ar[r, hookrightarrow] & \L_3(A) \ar[r] & \K_0(A)[2] \subseteq \K_0(A)
\end{tikzcd}\]
where the left most bottom arrow is surjective by \cref{prop:boundary-maps}. Consequently, the map $\L_3^h(A) \to \L_3(A)$ is injective. The claim then follows from the isomorphism $\L_3(A) \cong \K_7(A)$ of \cref{Thm:natural-ThmB} and the fact established in \cref{prop:boundary-maps} under this isomorphism, the map $\L_3(A) \to \K_0(A)$ appearing in the above diagram as the right most bottom horizontal map is given by multiplication by $\eta$.
\end{proof}

Finally, we say as much as we can about $\L_2^h(A)$:
\begin{Prop}\label[Prop]{prop:free-L2}
Let $A$ be a unital $C^*$-algebra. Then there is an exact sequence
\[ C(A)[2] \lto \K_2(A)/\eta \stackrel{\tilde{x}}{\lto} \L_2^h(A) \lto C(A)/2 \stackrel{\eta}{\lto} \K_1(A) \]
where $\tilde{x}$ is a map whose composition with the canonical map $\L_2^h(A) \to \L_2(A) \cong \K_6(A)[\eta]$ is given by multiplication by $x$.
\end{Prop}
\begin{proof}
We inspect the long exact sequence associated to the fibre sequence of \cref{prop:free-L} and use that $\L_1(SA) \cong \K_2(A)/\eta$ and $\L_0(SA) \cong \K_1(A)$ by \cref{Thm:natural-ThmB}. To see the claim about the composite of $\tilde{x}$ with the map $\L_2^h(A) \to \L_2(A) \cong \K_6(A)[\eta]$, we note that again by naturality, this composite is given by a multiple of the $x$ multiplication. The case $A= \C$ then shows the claim.
\end{proof}

\begin{Rmk}
The sequence of \cref{prop:free-L2} can of course simplify: For instance, if $C(A)$ has odd order, or when $C(A)[2] = 0$ and $0 \neq \eta \in \K_1(A)$, we find that $\L_2^h(A) \cong \K_2(A)/\eta$.
\end{Rmk}

\begin{Rmk}
The map $C(A)[2] \to \K_2(A)/\eta$ appearing in the sequence of \cref{prop:free-L2} picks out a particular element of the target (recall that $C(A)[2]$ is either cyclic of order 2 or trivial). Under the isomorphism 
\[\K_2(A)/\eta \cong \ker\big(\K_0(A_\C) \to \K_0(A)\big) \]
induced by the Wood sequence, this element is given by the composite
\[ C(A)[2] \lto \K_0(A)[2] \lto \ker(\K_0(A_\C)\to \K_0(A))\]
where we claim that the latter map is induced the canonical map $\K_0(A) \to \K_0(A_\C)$ (which, when restricted to 2-torsion lands in the indicated kernel since the composite $\K_0(A) \to \K_0(A_\C) \to \K_0(A)$ is given by multiplication by 2). Indeed, this map induces a natural map
\[ \K_1(A/2) \lto \K_0(A)[2] \lto \ker(\K_0(A_\C) \to \K_0(A)) \subseteq \K_0(A_\C)\]
which in turn determines the map in question, since the first map is surjective. This composite is determined by an element of $\K_0(\C/2) \cong \Z/2$, since the source is corepresented by $\R/2$. It then suffices to note that the map in question and the proposed map are both natural and non-trivial. To see that the map $C(A)[2] \to \K_2(A)/\eta$ appearing in \cref{prop:free-L2} is non-trivial, we can consider the case $A= \mathscr{O}_3^\C$. It satisfies $\K_0(A) = \Z/2$ and $\widetilde{\K}_0(A) = 0$. In particular $C(A) = \Z/2$ and the map $\L^h(A) \to \L(A)$ is an equivalence. Since $\K_1(A) = 0$ we deduce that $\L_1(A) \cong \L_3^h(A) = 0$. This shows that the map $C(A)[2] \to \K_2(A)/\eta$ is injective (and therefore in fact bijective). The same example also shows that the map $\K_0(A)[2] = \K_0(A) \to \K_0(A_\C)$ is non-trivial: Indeed, since $A$ is complex there is an isomorphism $A_\C \cong A \times A$ under which the map from $A$ corresponds to the diagonal.
\end{Rmk}

\begin{Rmk}
Finally, we explain that for a unital $C^*$-algebra $A$, the map $\tau \colon \k(A) \to \L(A)$, on positive homotopy groups factors canonically through the canonical map $\L^h(A) \to \L(A)$.
Indeed, we shall argue that there is a canonical map $\k^\free(A) \to \L^h(A)$, where $\k^\free(A)$ denotes the group completion of the topological category of free $A$-modules, participating in the following commutative diagram. 
\[\begin{tikzcd}
	\k^\free(A) \ar[r] \ar[d] & \L^h(A) \ar[d] \\
	\k(A) \ar[r] & \L(A)
\end{tikzcd}\]
The left vertical map is induced by the canonical inclusion of free into projective modules and induces an equivalence on connected covers. Indeed, there is a commutative diagram 
\[\begin{tikzcd}
	\k^\free(A) \ar[r] \ar[d] & \Gw^\free_\top(A) \ar[r] \ar[d] & \L^h(A) \ar[d] \\
	\k(A) \ar[r] & \Gw_\top(A) \ar[r] & \L(A)
\end{tikzcd}\]
where the maps from K-theory to Grothendieck--Witt theory equip a module over $A$ with its unique positive definite form as described earlier.

Under the isomorphisms of \cref{Prop:free-L-theory-degree1+3}, the map $\k^\free \to \L^h$ on low homotopy groups is then given as follows:
\begin{enumerate}
\item the canonical projection $\K_1(A) \to \K_1(A)/\langle \eta \rangle \cong \L_1^h(A)$, 
\item the map $\K_2(A) \to \K_2(A)/\eta \stackrel{\tilde{x}}{\to} \L_2^h(A)$ where $\tilde{x}$ is as in \cref{prop:free-L2}, and
\item the map $\K_3(A) \to \K_7(A) \times_{\K_8(A)} C(A) \cong \L_3^h(A)$ given by the $x$-multiplication -- note that $\eta x = 0$, so the $x$ multiplication on $\K_3$ indeed has image in the claimed subgroup of $\K_7(A)$.

\end{enumerate}

\end{Rmk}

\section{Examples}\label{Sec:examples}
In this section, we present a number of examples where we calculate L-groups of $C^*$-algebras. We note here, that due to the fact that the (graded) cohomological dimension of $\pi_*(\L(\R))$ is 1, an $\L(\R)$-module spectrum is (non-canonically) determined by its homotopy groups. Below, we shall therefore concentrate on calculating L-groups, and sometimes construct in addition canonical fibre sequences describing the L-spectrum. To do so efficiently, we begin with the following lemma.
\begin{Lemma}\label[Lemma]{Lemma:preparation}
We have a canonical equivalence $\Z \otimes_\ko \L(\R) \cong \L(\C)/2$.
\end{Lemma}
\begin{proof}
There is a canonical fibre sequence $\Sigma^2\ku \stackrel{\beta}{\to} \ku \lto \Z$. Therefore, applying $- \otimes_\ko \L(\R)$ and \cref{ThmA}, there is a canonical fibre sequence
\[ \Sigma^2 \L(\C) \stackrel{2b_\C}{\lto} \L(\C) \lto \Z \otimes_\ko \L(\R),\]
where $b_\C \in \L_2(\C)\cong \Z$ denotes the generator such that $\tau(\beta) = 2b_\C$, see \cite[Lemma 4.9]{LN}.
Since $b_\C$ is invertible in $\L_*(\C)$, the lemma follows. 
\end{proof}

\begin{example}
Let $A$ be a complex $C^*$-algebra. Then 
\[ \L(A) \simeq \k(A) \otimes_\ko \L(\R) \simeq \k(A) \otimes_\ku \ku \otimes_\ko \L(\R) = \k(A) \otimes_\ku \L(\C).\]
Furthermore, we recover that $\L_0(A) \cong \K_0(A)$ and $\L_1(A) \cong \K_1(A)$, since $\eta$ is trivial on $\ku$-modules.
\end{example}

\begin{example}\label[example]{ex:shifts-of-R}
We have $\Sigma^n \L(\R) \stackrel{\simeq}{\to} \L(\R[n])$ for $0 \leq n \leq 3$. Here, the notation $A[n]$ refers to the $n$-fold suspension of $A$ in the stable $\infty$-category $\KK$; we again note that this construction is implemented by an appropriate $C^*$-algebraic suspension, that is, $A[-1]$ is represented by $SA$. Indeed, the example follows from the fibre sequence
\[ \Sigma \L(A[-1]) \lto \L(A) \lto \K_0(A)^{tC_2} \]
obtained in \cite[Prop.\ 4.6]{LN} and the fact that $\K_n(\R) = 0$ for $n=5,6,7$. We note that the proof of \cite[Prop.\ 4.6]{LN} applies verbatim to real $C^*$-algebras, now that we know that $\L$-theory factors through $\KK$ also for real $C^*$-algebras.
\end{example}

\begin{example}
Let $\H$ be the quaternions. Then we have $\k(\H) = \ksp \simeq \tau_{\geq 0} \Omega^4 \ko$, and hence $\ell(\H) \simeq \ksp\otimes_\ko \ell(\R)$. Consequently, the L-groups are given by 
\begin{enumerate} 
	\item $\L_0(\H) = \K_0(\H) \cong \Z$,
	\item $\L_1(\H) = \coker(\K_0(\H) \to \K_1(\H)) = \coker(\Z \to 0) = 0$,
	\item $\L_2(\H) = \ker(\K_6(\H) \to \K_7(\H)) = \ker(\Z/2 \to 0) = \Z/2$, and
	\item $\L_3(\H) = \K_7(\H) = 0$.
\end{enumerate}
In addition, from the fibre sequence $\Sigma^4 \ko \to \ksp \to \Z$ and \cref{Lemma:preparation}, we obtain a canonical fibre sequence
\[  \L(\R) \lto \L(\H) \lto \L(\C)/2,\]
where the first map classifies $2$ times a generator of $\L_0(\H)$.
\end{example}

\begin{example}\label[example]{ex:final-shifts-of-R}
Since $\H \simeq \R[4]$ in $\KK$, we have already determined $\L(\R[n])$ for $0 \leq n \leq 4$. In addition, similarly as in \cref{ex:shifts-of-R}, we have that $\Sigma \L(\H) \simeq \L(\H[1]) \simeq \L(\R[5])$ since $\K_3(\R) = 0$. Since $\R[n] \simeq \R[n+8]$ in $\KK$ by real Bott periodicity, we shall now also calculate the L-groups of the remaining shifts of $\R$, namely $\R[6]$ and $\R[7]$. Here, we find
\begin{enumerate}
\item $\L_0(\R[6]) = \Z/2$ and $\L_0(\R[7]) = \Z/2$,
\item $\L_1(\R[6]) = 0$ and $\L_1(\R[7]) = 0$,
\item $\L_2(\R[6]) = \Z$ and $\L_2(\R[7]) = 0$, and
\item $\L_3(\R[6]) = \Z/2$ and $\L_3(\R[7]) = \Z$.
\end{enumerate}

\end{example}

\begin{example}
As spectra, there is an equivalence $\L(\C[1]) \simeq \Sigma \L(\C)$. However, the canonical map $\Sigma \L(\C) \to \L(\C[1])$ is \emph{not} an equivalence, as follows again from \cite[Proposition 4.6]{LN}: its cofibre is given by $\Z^{tC_2}$. In other words, the canonical map identifies with the times 2 map on $\Sigma \L(\C)$.
\end{example}

\begin{example}
Consider the algebra $C(\mathbb{T})$ of continuous real valued functions on the circle. Then we have an equivalence in $\KK$ given by $C(\mathbb{T}) = \R \oplus \R[-1]$. From this, the fact that L-theory preserves products, and \cref{ex:final-shifts-of-R} we obtain the following L-groups.
\begin{enumerate}
\item $\L_0(C(\mathbb{T})) = \Z \oplus \Z/2$, 
\item $\L_1(C(\mathbb{T})) = 0$,
\item $\L_2(C(\mathbb{T})) = 0$,
\item $\L_3(C(\mathbb{T})) = \Z$.
\end{enumerate}
\end{example}

\begin{example}
Let $G$ be a torsion-free group for which the Baum--Connes conjecture holds. Then we get 
\begin{enumerate} 
	\item $\L_0(C^*_rG) = \KO_0(BG)$,
	\item $\L_3(C^*_rG) = \KO_{-1}(BG)$,
\end{enumerate}
By Anderson duality, this shows that one can recover the abelian group $\KO^4(BG)$ from $\L_*(C^*_rG)$. Indeed, there is a short exact sequence
\[ 0 \lto \Ext^1_\Z(\KO_{-1}(BG),\Z) \lto \KO^4(BG) \lto \Hom_\Z(\KO_0(BG),\Z) \lto 0 \]
which splits non-canonically; see e.g.\ \cite{HLN} for a review of Anderson duality and the fact that the Anderson dual of $\KO$ is given by $\Omega^4 \KO$.
\end{example}

\begin{example}
The Baum--Connes conjecture is known to be true for free groups $\mathscr{F}_n$ of rank $n\geq 1$. In particular, we have an equivalence $\K(C^*_r \mathscr{F}_n) \simeq \KO \oplus \bigoplus_n \KO[1]$, and therefore obtain 
\[ \L(C^*_r \mathscr{F}_n) \simeq \L(\R) \oplus \bigoplus_n \L(\R)[1].\]
\end{example}

\begin{Rmk}
We warn the reader that, contrary to the complex case, $C^*_r\Z$ is not isomorphic to $C(\mathbb T)$, but rather to the algebra of $C_2$-equivariant continuous functions $\mathbb{T} \to \C$, where $\mathbb{T}$ and $\C$ both carry the complex conjugation action.
\end{Rmk}

\begin{example}
Let $\Sigma_g$ be an orientable surface of genus $g \geq 1$ and $\pi$ its fundamental group. The Baum--Connes conjecture is known for surface groups, so we find that 
\[\K(C^*_r \pi) \simeq \Sigma_{g,+} \otimes \KO \simeq \KO \oplus \KO[1]^{\oplus 2g} \oplus \KO[2].\]
Consequently, we obtain 
\[ \L(C^*_r \pi) \simeq \L(\R) \oplus \L(\R) \oplus \L(\R)[1]^{\oplus 2g} \oplus \L(\R)[2].\]
\end{example}

\begin{example}
Let $W$ be a right angled Coxeter group associated to a flag complex $\Sigma$ as studied e.g.\ in \cite{KLL}. Then \cite[Theorem 7.16]{KLL} shows that $\K(C^*_rW) \simeq \KO^r$, where $r$ is the number of simplices of $\Sigma$, including the empty simplex. Therefore, we find $\L(C^*_r W) \simeq \L(\R)^r$.

\end{example}

\begin{Rmk}
In fact, \cite{KLL} shows that there are explicit maps $\alpha_i \colon \R \to \R W \subseteq C^*_rW$, for $i = 1,\dots,r$, such that the induced maps 
\[ \bigoplus_r \KO \lto \K(C^*_r W) \text{ and } \bigoplus_r \L(\R) \lto \L(\R W) \]
are equivalences. Combined with \cref{ThmA}, we deduce that among the following two maps
\[ \bigoplus_r \L(\R) \lto \L(\R W) \lto \L(C^*_r W) \]
both the first map and the composite are equivalences. Therefore, so is the second map. This gives a non-trivial example where the completion conjecture of \cite[Conjecture after Cor.\ 5.7]{LN}, which states that the map $\L(\R G) \to \L(C^*_rG)$ is an equivalence after inverting 2, is in fact true \emph{without} inverting 2. 
\end{Rmk}

\begin{example}
Let $A$ be a real $C^*$-algebra equipped with an automorphism $\varphi \colon A \to A$. Then there is a fibre sequence 
\[ A \stackrel{\id-\varphi}{\lto} A \lto A \rtimes_\varphi \Z \]
in $\KK$, this is essentially the Pimsner--Voiculesu sequence in $\KK$-theory, see e.g.\ \cite[\S 19.6]{Blackadar}. If $\id-\varphi_*$ is injective on $\K_{-1}(A)$, one also obtains a fibre sequence 
\[ \k(A) \stackrel{1-\varphi_*}{\lto}  \k(A) \lto \k(A \rtimes_\varphi \Z)\]
of $\ko$-modules. Consequently, there is then also a fibre sequence
\[ \L(A) \stackrel{1-\varphi}{\lto} \L(A) \lto \L(A \rtimes_\varphi \Z).\]
More generally, there is a similar (conditional) fibre sequence describing the L-theory of reduced crossed products by free groups. 
\end{example}

\begin{example}
An example of a crossed product by $\Z$ is the real rotation algebra $A_\theta = C(\mathbb{T}) \rtimes_\theta \Z$ where $\theta$ is a real number and acts on functions on the circle group $\mathbb{T}$ by a rotation by $\theta$. However, by homotopy invariance, in $\KK$ we have an equivalence $A_\theta \simeq C(\mathbb T) \oplus C(\mathbb T)[1] \simeq \R \oplus \R[-1] \oplus \R[1] \oplus \R$. We therefore obtain
\[ \L(A_\theta) \simeq \L(\R)^{\oplus 2} \oplus \L(\R[-1]) \oplus \L(\R[1]).\]
Using our previous calculations, we finally obtain
\begin{enumerate}
\item $\L_0(A_\theta) = \Z^2\oplus \Z/2$,
\item $\L_1(A_\theta) = \Z$,
\item $\L_2(A_\theta) = 0$, and
\item $\L_3(A_\theta) = \Z$.
\end{enumerate}
We note that the order structure on $\K_0(A)$ gives additional information (also on $\theta$), and that by the isomorphism $\K_0(A) \cong \L_0(A)$, this order structure is also present in L-theory. We are not aware of a description of the order structure on L-theory which does not make use of the isomorphism to K-theory.
\end{example}

\begin{example}\label[example]{Ex:L-weaker-than-K}
We consider the real Cuntz algebras $\mathscr{O}^\R_{n+1}$. In $\KK$, there is a canonical equivalence $\mathscr{O}^\R_{n+1} \simeq \R/n$. In particular, $\k(\mathscr{O}^\R_{n+1}) \simeq \ko/n$. Therefore, we find 
\[ \L(\mathscr{O}^\R_{n+1}) \simeq \ko/n \otimes_\ko \L(\R) = \L(\R)/n.\]
Likewise, one can consider the tensor poducts $\mathscr{O}^\R_{n+1} \otimes_\R \mathscr{O}^\R_{m+1}$, where one finds 
\[ \L(\mathscr{O}^\R_{n+1} \otimes_\R \mathscr{O}^\R_{m+1}) \simeq \L(\R)/\gcd(m,n) \oplus \Sigma \L(\R)/\gcd(m,n)\]
contrary to the case of K-theory, where the real K-groups of $\mathscr{O}^\R_{n+1} \otimes_\R \mathscr{O}^\R_{m+1}$ do not only depend on the greatest common divisor of $m$ and $n$ -- simply because $(\ko/n)/m$ does not only depend on this number; see \cite{Boersema1} for explicit calculations.

Therefore, not surprisingly, L-theory of \emph{real} $C^*$-algebras is a strictly weaker invariant than K-theory: The real $C^*$-algebras $\mathscr{O}^\R_{3} \otimes_\R \mathscr{O}^\R_5$ and $\mathscr{O}^\R_{3} \otimes_\R \mathscr{O}^\R_3$ are distinguished by their K-groups, but not by their L-groups.
\end{example}

\begin{example}
Let $\mathscr{E}^\R_{2n}$ denote the simple separable nuclear real form of the Cuntz-algebra $\mathscr{O}^\C_{2n+1}$ considered in \cite{Boersema2}. Its topological K-theory is given by
\[ \K(\mathscr{E}^\R_{2n}) \simeq \KO/nx \]
where $x \in \pi_4(\KO)$ is a generator. Note that its complexification is given by $\KU/2n\beta^2 \simeq \KU/2n$, compatible with the equivalence $\K(\mathscr{O}^\C_{2n+1}) \simeq \KU/2n$. There is a fibre sequence
\[ \Z \lto \ko/nx \lto \k(\mathscr{E}^\R_{2n}) \]
where the first map induces multiplication by $4n$ on $\pi_0$. We conclude that there is a fibre sequence
\[ \L(\C)/2 \lto \L(\R)/8n \lto \L(\mathscr{E}^\R_{2n}) \]
in which the first map induces the non-zero map on $\pi_{4i}$.
In particular, the L-groups are given by 
\begin{enumerate}
\item $\L_0(\mathscr{E}^\R_{2n}) = \Z/4n$,
\item $\L_1(\mathscr{E}^\R_{2n}) = 0$,
\item $\L_2(\mathscr{E}^\R_{2n}) = \Z/2$,
\item $\L_3(\mathscr{E}^\R_{2n}) = \Z/2$.
\end{enumerate}

\end{example}

\begin{example}
We end with a number of structural examples:
\begin{enumerate}
\item The L-groups of a $C^*$-algebra with $\K_{1}(A_\C) = \K_{7}(A) = \K_{6}(A) = 0$, are concentrated in degrees divisible by $4$. 
\item If the K-groups of a $C^*$-algebra are finitely generated, then so are the L-groups.
\item If the $\K$-groups of a $C^*$-algebra vanish after inverting a number $n$ (e.g.\ when they are $n$-primary torsion), then so do the L-groups.
\item In \cref{Ex:L-weaker-than-K} we have seen that there are algebras which cannot be distinguished by their L-theories, but by their K-theories. However, we record here that a \emph{map} $f \colon A \to B$ of $C^*$-algebras induces an equivalence on K-theory if and only if it does so on L-theory: Indeed the only if part follows immediately from \cref{ThmA}, so let us argue the if part. By passing to the cofibre of the map associated to $f$ in $\KK$, we may equivalently show that $\L(A) = 0$ implies $\K(A) = 0$. By \cref{ThmB}, we deduce from $\L(A) = 0$ that $\K_n(A) = 0$ for $n=-2,-1,0,1$. By the generalized Wood sequence (see the Remark after \cref{ThmB}), we deduce that $\K(A_\C) = 0$ (the generalized Wood sequence reveals that $\K_{1}(A_\C) = 0 = \K_0(A_\C)$) and therefore that $\eta\colon \Sigma \K(A) \to \K(A)$ is an equivalence, and consequently $\eta^n$ is also an equivalence for any $n\geq 1$. However, $\eta^3 = 0$, showing that $\K(A)$ must be zero.
\end{enumerate}
\end{example}

\section{Integral Baum--Connes and Farrell--Jones comparison}\label{Sec_seven}

In \cite{LN}, we have used the equivalence $\K\adjt \simeq \L\adjt$ to compare the Baum--Connes assembly map and the L-theoretic Farrell--Jones assembly maps after inverting 2. The purpose of this section is to prove the following integral refinement of this result.
\begin{Thm}\label[Thm]{Thm:analog-ThmD}
The map $\tau \colon \k \to \L$ induces the commutative diagram
\[\begin{tikzcd}
	\ko_*^G(\underline{E}G) \ar[rr,"\mathrm{BC}"] \ar[d,"\tau"] & & \k_*(C^*_rG) \ar[d,"\tau"] \\
	\L\R^G_*(\underline{\underline{E}}G) \ar[r,"\mathrm{FJ}"] & \L_*(\R G) \ar[r] & \L_*(C^*_r G)
\end{tikzcd}\]
\end{Thm}

In order to explain the terms in the theorem, we will first briefly recall the setup for assembly maps as proposed by Davis and L\"uck \cite{DL}.

The Davis--L\"uck picture for assembly maps starts with a discrete group $G$ and an equivariant homology theory $E$, encoded as a functor $\Orb(G) \to \Sp$, where $\Orb(G)$ is the \emph{orbit category} of $G$, that is, the full subcategory of $G$-sets consisting of transitive $G$-sets (i.e.\ $G$-sets isomorphic to $G/H$ for a subgroup $H$ of $G$). A family $\cF$ of subgroups of $G$ is a collection of subgroups closed under conjugation and passing to further subgroups.
Associated to any such family $\cF$, we may form the $\cF$-orbit category $\Orb_\cF(G)$ which is the full subcategory on all transitive $G$-sets whose stabilisers belong to $\cF$ (i.e.\ $G$-sets isomorphic to $G/H$ for $H \in \cF$).
The $\cF$-assembly map for $G$ and $E$ is then given by the canonical map 
\[ E^G(E_\cF G) \stackrel{\mathrm{def}}{=} \colim\limits_{G/H \in \Orb_\cF G} E(G/H) \lto E(G/G).\]

Given an inclusion of families $\cF \subseteq \cF'$, there is an evident factorisation of the $\cF$-assembly map as follows:
\[ E^G(E_\cF G) \lto E^G(E_{\cF'}G) \lto E(G/G)\]
in which the first map is referred to as the \emph{relative} assembly map (with respect to the inclusion of families $\cF \subseteq \cF')$ and the second map is the $\cF'$-assembly map. Following standard notation we shall also write $E_{\Fin} G = \underline{E}G$ and $E_{\Vcyc}G = \underline{\underline{E}}G$.

Relevant for us will be the functors given by equivariant topological K-theory and equivariant L-theory. These are functors 
\[ \K^G,\L R^G \colon \Orb(G) \lto \Sp \quad,\quad G/H \mapsto K(C^*H), \L(R H)\]
for an involutive ring $R$, see e.g.\ \cite{LN} for further details. For the family of finite subgroups $\Fin$, we shall denote these assembly maps by 
\[ \BC \colon \K^G_*(\underline{E}G) \lto \K_*(C^* G) \quad \text{ and } \quad \FJ \colon \L R^G_*(\underline{E}G) \lto \L(RG).\]
We also note that the map $\FJ$ factors as the composite
\[ \L R^G_*(\underline{E}G) \lto \L R^G_*(\underline{\underline{E}}G) \lto \L(RG),\]
whose second map we shall later also denote by $\FJ$.

\begin{Rmk}
We warn the reader that the L-theoretic Farrell--Jones conjecture is more specifically about the assembly map for the family $\Vcyc$ of virtually cyclic subgroups and for a related (but in general different) functor $G/H \mapsto \LL^\q(R H)$ where $\LL^\q$ is the Karoubi-localisation of $\L^\q$ in the sense of \cite{CDHIV}, also known as \emph{universally decorated} L-theory, denoted by $\L^{\langle -\infty \rangle,\q}$ in the literature, where the superscript $\q$ refers to quadratic rather than symmetric L-theory. We show in \cref{Thm:assembly-symmetric-L} below, that for a regular ring $R$ and a torsion free group $G$, the Farrell--Jones conjecture implies that the map denoted $\FJ$ above is also an isomorphism (in fact, for either of the two maps denoted FJ above, as the first map in the composite is an isomorphism under the assumptions made, see \cref{Prop:relative-assembly}); to the best of our knowledge, this had not been observed so far. 
\end{Rmk}

\begin{Rmk}
The assembly map in topological K-theory described above is related to the Baum--Connes conjecture which was originally phrased in terms of equivariant Kasparov theory. First and foremost, this conjecture is more specifically about the composite
\[ \K_*^G(\underline{E}G) \stackrel{\BC}{\lto} \K_*(C^*G) \lto \K_*(C^*_r G) \]
where $C^*G \to C^*_r G$ is the canonical morphism. We note here that the association $G \mapsto C^*_r G$ is not functorial in group homomorphisms, as famously the reduced group $C^*$-algebra $C^*_r \mathscr{F}_n$ of a non-abelian free group is a simple algebra \cite{Powers}. Therefore, the Davis--L\"uck picture for the assembly map in topological K-theory uses the full group $C^*$-algebra instead.

Now, it was shown in \cite{Kranz} (and later and with different methods in \cite{BEL}) that the assembly map for $G$, the family $\Fin$ of finite subgroups of $G$, and the functor $\K^G$ is isomorphic to the Baum--Connes assembly map, and in \cite{Land} that for torsion-free groups, this assembly map has an interpretation as taking a Mishchenko--Fomenko index (this latter result was a folklore result known to the experts for a long time).

In addition, the assembly map in topological K-theory is often performed using the complex group $C^*$-algebra rather than the real one, but see \cite{Schick} for a relation between the two which says that there is a comparison square for the real and complex assembly maps, and that either of these two assembly maps is an isomorphism (in all degrees) if and only if the other is.
\end{Rmk}

The natural map $\k \to \L$ as functors on the category $\KK$ induces a natural transformation $\k^G \to \L\R^G$ as functors on the orbit category, see \cite{LN} for the details. Consequently, we obtain the following theorem, which is an integral analog of \cite[Theorem D]{LN}.
\begin{Thm}
The map $\tau \colon \k \to \L$ of \cref{Thm:existence+uniqueness-of-tau} induces the following commutative diagram.
\[\begin{tikzcd}
	\ko_*^G(\underline{E}G) \ar[rr,"\mathrm{BC}"] \ar[d,"\tau"] & & \k_*(C^*G) \ar[d,"\tau"] \\
	\L\R^G_*(\underline{\underline{E}}G) \ar[r,"\mathrm{FJ}"] & \L_*(\R G) \ar[r] & \L_*(C^*G)
\end{tikzcd}\]
\end{Thm}

\begin{proof}[Proof of \cref{Thm:analog-ThmD}]
There is a canonical map $C^*G \to C^*_r G$ from the full to the reduced group $C^*$-algebra. Since $\tau$ is natural, we obtain the commutative diagram
\[\begin{tikzcd}
	\ko_*^G(\underline{E}G) \ar[rr,"\mathrm{BC}"] \ar[d,"\tau"] & & \k_*(C^*G) \ar[d,"\tau"] \ar[r] & \k_*(C^*_r G) \ar[d,"\tau"]  \\
	\L\R^G_*(\underline{\underline{E}}G) \ar[r,"\mathrm{FJ}"] & \L_*(\R G) \ar[r] & \L_*(C^* G) \ar[r] & \L_*(C^*_r G)
\end{tikzcd}\]
which is the content of \cref{Thm:analog-ThmD}.
\end{proof}

\begin{Rmk}
Upon inverting $2$ and the Bott element $\beta_\R$, the diagram of \cref{Thm:analog-ThmD} becomes equivalent to the diagram
\[\begin{tikzcd}
	\KO^G_*(\underline{E}G)\adjt \ar[rr] \ar[d,"\cong"] & & \KO_*(C^*_r G)\adjt \ar[d,"\cong"] \\
	\L\R^G_*(\underline{E}G)\adjt \ar[r] & \L_*(\R G)\adjt \ar[r] & \L_*(C^*_r G)\adjt
\end{tikzcd}\]
which is the one obtained earlier in \cite[Theorem D]{LN}. This uses in particular that the comparison map $\L\R^G_*(\underline{E}G) \to \L\R^G_*(\underline{\underline{E}}G)$ is an isomorphism after inverting 2 \cite[Proposition 2.18]{Lueck-assembly}.
\cref{Thm:analog-ThmD} in addition provides some finer information about the comparison, as for instance the kernels and cokernels of the vertical maps appearing in the diagram of \cref{Thm:analog-ThmD} can be analysed by means of \cref{Rem:induced-map}
\end{Rmk}

To put \cref{Thm:analog-ThmD} into context, we note that the diagram in it participates in the following larger diagram:
\[\begin{tikzcd}
\KO_*^G(\underline{E}G) \ar[rr,"\mathrm{BC}"] & & \K_*(C^*_r G)  \\
\ko_*^G(\underline{E}G) \ar[rr] \ar[u] \ar[d,"\tau"] & & \k_*(C^*_rG) \ar[u,"\simeq_{\geq 0}"'] \ar[d] \\
\L\R_*^G(\underline{\underline{E}}G) \ar[r,"\FJ_\R"] & \L_*(\R G) \ar[r]  & \L_*(C^*_rG) \\ 
\L\Z_*^G(\underline{\underline{E}}G) \ar[r,"{\mathrm{FJ}_\Z}"] \ar[u] & \L_*(\Z G) \ar[u] \\
\end{tikzcd}\]
The map labelled $\tau$ in the diagram factors as 
\[ \ko_*^G(\underline{E}G) \lto \L\R^G_*(\underline{E}G) \lto \L\R^G_*(\underline{\underline{E}}G) \]
where the second map is split injective, see the argument below, and the first map is in principle understandable by means of \cref{ThmA} \& \ref{ThmB}. We shall argue that the full Farrell--Jones conjecture for $G$ implies that\footnote{That is, we assume that $G$ is a Farrell--Jones group in the sense of \cite[\S 5.2]{HLLRW}.} 
\begin{enumerate}
	\item\label{item1} the map $\FJ_\R$ is an isomorphism, and
	\item\label{item2} the map $\FJ_\Z$ is an isomorphism if $G$ is torsion free. 
\end{enumerate}
To see statement \eqref{item1} and the claim about the split injectivity above, we first note that there is a commutative diagram
\[\begin{tikzcd}
	\L\R^G(\underline{E}G) \ar[r] \ar[d] & \L\R^G(\underline{\underline{E}}G) \ar[r] \ar[d] & \L(\R G) \ar[d] \\ 
	\LL\R^G(\underline{E}G) \ar[r] & \LL\R^G(\underline{\underline{E}}G) \ar[r] & \LL(\R G)
\end{tikzcd}\]
where for a ring $R$, $\LL(R)$ is what is denoted by $\L^{\langle -\infty \rangle}(R)$ in the literature, see \cite[Remark 1.21]{Lueck-assembly} and \cite{CDHIV} for the notation. There is a canonical map $\L(R) \to \LL(R)$ which is an equivalence for instance if $K(R)$, the algebraic $K$-theory of $R$, is connective, see e.g.\ \cite[Remark 1.22]{Lueck-assembly}.
We claim that the vertical maps in the above diagram are all equivalences: 
Indeed, the K-theoretic Farrell--Jones conjecture together with the fact that $\R$ is a regular $\Q$-algebra and \cite[Proposition 2.14]{Lueck-assembly} implies that $K(\R G)$ is connective, so that $\L(\R G) \to \LL(\R G)$ is an equivalence, see \cite[Conjecture 3.3]{Lueck-assembly}. To see that also middle and left vertical maps are equivalences, we use the same argument for 
$G$ replaced by virtually cyclic subgroups and finite subgroups of $G$, respectively - note here that the class of group for which the (full) Farrell--Jones conjectures hold is closed under taking subgroups.

Now, the Farrell--Jones conjecture implies that the right lower horizontal map is an equivalence, and \cite[Proposition 2.16]{Lueck-assembly} states that the left lower horizontal map is split injective.

Statement \eqref{item2} requires different methods, since the Farrell--Jones conjecture is a conjecture about \emph{quadratic} L-theory whereas we make a statement about \emph{symmetric} L-theory. Note that this subtlety does not appear for group rings over $\R$ since there, quadratic and symmetric L-theory agree. We give a proof of statement \eqref{item2} in \cref{Thm:assembly-symmetric-L} below, relying on some recent developments in hermitian K-theory.
\newline

The big diagram appearing above simplifies if the group $G$ is torsion free, see also \cref{Prop:relative-assembly} below: In this case one obtains the following diagram.
\[\begin{tikzcd}
	BG\otimes \KO \ar[rr,"\mathrm{BC}"] & & \K(C^*_rG) \\
	BG\otimes \ko \ar[rr] \ar[u,"\simeq_{\geq \mathrm{dim} BG}"] \ar[d] & & \k(C^*_r G) \ar[u,"\simeq_{\geq 0}"'] \ar[d] \\
	BG\otimes \L(\R) \ar[r,"\FJ_\R"] & \L(\R G) \ar[r] & \L(C^*_r G) \\
	BG\otimes \L(\Z) \ar[r,"\FJ_\Z"] \ar[u] & \L(\Z G) \ar[u]
\end{tikzcd}\]

In addition, for torsion free groups, the Farrell--Jones conjecture in quadratic L-theory implies the one in symmetric L-theory, see the subsection below. However, the lower vertical comparison maps which change the base ring in the L-theoretic FJ conjecture from $\Z$ to $\R$ are quite subtle to analyse, in particular integrally, but even after inverting 2, see e.g.\ \cite[Remark 3.20]{Lueck-assembly}. If furthermore $BG$ has an $n$-dimensional classifying space, then the left top most vertical map is an equivalence in degrees $\geq n+1$.

\subsection*{Farrell--Jones for symmetric L-theory}
In this section we prove the following result we have indicated above and which might be of some independent interest.
\begin{Thm}\label[Thm]{Thm:assembly-symmetric-L}
Let $G$ be a torsion free group and $R$ a regular ring. Assume that the FJ conjecture holds for $G$. Then the assembly map
\[ BG \otimes \L^\s(R) \lto \L^\s(RG) \]
is an equivalence.
\end{Thm}

To connect it more explicitly to the statement \eqref{item2} above, we also record here the following result.
\begin{Prop}\label[Prop]{Prop:relative-assembly}
Let $R$ be a regular ring and $G$ a torsion free group. Then the relative assembly map 
\[ BG \otimes \L^\s(R) \lto \L^\s R^G(\underline{\underline{E}}G) \]
is an equivalence.
\end{Prop}
\begin{proof}
By the transitivity principle for assembly maps \cite[Theorem 2.9]{Lueck-assembly}, we need to show that for each virtually cyclic subgroup $V$ of $G$, the assembly map $BV \otimes \L^\s(R) \to \L^\s(RV)$ is an equivalence. Now, since $G$ is torsion free, so is $V$, and therefore $V$ is either trivial or isomorphic to $\Z$ \cite[Lemma 2.15]{Lueck-assembly}. We therefore need to show that the \emph{Shaneson--Ranicki splitting} holds for symmetric L-theory, which is for instance done in the generality of bordism invariant Verdier localising invariants of Poincar\'e categorie in \cite{CDHIV}, see \cite{MR} for an earlier proof of the splitting result for symmetric L-theory. Note that we also use that $\K_0(RV)\cong \K_0(R)$ in order to ensure that no decoration problems appear in the Shaneson--Ranicki splitting.
\end{proof}

In what follows, we will freely make use of the language and notation developed in the sequence of papers \cite{CDHI,CDHII,CDHIII,CDHIV}. Suffice it to say here that for a space\footnote{here, best to be thought of as an $\infty$-groupoid} $X$ and a Poincar\'e category $(\cC,\QF)$, we write 
\[ (\cC,\QF)_X = \colim\limits_X (\cC,\QF) \]
for the tensor of $(\cC,\QF)$ with $X$. We call $(\cC,\QF)_X$ the \emph{visible} Poincar\'e category associated to $X$ and $(\cC,\QF)$, see \cite{CDHI} for some explanation of the terminology and how its L-theory connects to previously studied versions of visible L-theory. In the proof of the following result, which we initially learned from Yonatan Harpaz, we will describe the Poincar\'e category $(\cC,\QF)_X$ in more detail.
\begin{Lemma}\label[Lemma]{lemma:general-pullback-L-theories}
Let $X$ be a space, $\cC$ be a stable $\infty$-category and $\QF \to \QF'$ a map of Poincar\'e structures on $\cC$ inducing an equivalence on the bilinear parts of $\QF$ and $\QF'$. Then the diagram
\[\begin{tikzcd}
	X \otimes \L(\cC,\QF) \ar[r] \ar[d] & X \otimes \L(\cC,\QF') \ar[d] \\
	\L((\cC,\QF)_X) \ar[r] & \L((\cC,\QF')_X)
\end{tikzcd}\]
is a pullback, where the vertical maps are the assembly maps.
\end{Lemma}
\begin{proof}
Let $T= \cofib(\QF \to \QF')$, which is by assumption an exact functor $\cC^\op \to \Sp$. It is therefore a filtered colimit of representables. All terms in the diagram appearing in the lemma preserve filtered colimits of Poincar\'e categories, so it suffices to prove the lemma in the case where $T$ is represented by an object $t$ of $\cC$, i.e.\ where $T = \map_\cC(-,t)$. In this case, $T_X = \cofib(\QF_X \to \QF'_X)$ is given as follows.

We recall that $\cC_X$ is the subcategory of $\Fun(X^\op,\Pro(\cC))$\footnote{Of course, $X^\op \simeq X$, but in order to get op's and colimits vs limits correct, it is best not to identify $X$ with $X^\op$ just yet.} generated under finite limits from the right Kan extensions of functors $\ast \to \cC \to \Pro(\cC)$ along inclusions $\ast \to X^\op$. We then have that for $\varphi \in \cC_X \subseteq \Fun(X^\op,\Pro(\cC))$
\[ T_X(\varphi) = \colim\limits_{x\in X} T(\varphi(x)) = \colim\limits_{x\in X} \map(\varphi(x),t) \]
so that $T_X$ is represented by the object $r^*(t)$ in $\Fun(X^\op,\cC) \subseteq \Fun(X^\op,\Pro(\cC))$, where $r\colon X \to \ast$ is the unique map. Now in general, $r^*(t)$ is not contained in $\cC_X$ (though it is the case if $X$ is compact to which the general case reduces again by using that all functors in sight preserve filtered colimits). Regardless, one can write it as a filtered colimit of objects in $\cC_X$.

The formula for relative L-theory of \cite{HNS} says that there is an equivalence
\[ \L(\cC;\QF,\QF') = \Eq\big( \QF'(Dt) \rightrightarrows B_\QF(Dt,Dt) \big),\]
where $D$ denotes the (common) duality of $\cC$ - here, one of the maps is the canonical forgetful map $\QF'(Dt) \to B_\QF(Dt,Dt)^{hC_2} \to B_\QF(Dt,Dt)$, and the other one is the canonical map $\QF'(Dt) \to B_\QF(Dt,Dt)^{hC_2} \to B_\QF(Dt,Dt)^{tC_2} \simeq \Lambda_{\QF^\s}(Dt) \to \Lambda_{T}(Dt)=\map_\C(Dt,t) \simeq B_\QF(Dt,Dt)$.
Likewise, we obtain 
\[ \L(\cC_X;\QF_X,\QF'_X) = \Eq\big( \QF'_X(D(r^*(t))) \rightrightarrows B_{\QF_X}(D(r^*(t)),D(r^*(t))) \big).\]
Now $D(r^*(t)) = r^*Dt$, and therefore, since $\QF_X(\varphi) = \colim_X \QF(\varphi(x))$, and likewise for the bilinear functor \cite[Prop.\ 6.4.3]{CDHI}, we find that 
\[ \L(\cC_X;\QF_X,\QF'_X) = \colim_X \Eq\big(\QF'(Dt) \rightrightarrows B_{\QF}(Dt,Dt) \big) = X\otimes \L(\cC;\QF,\QF') \]
and one checks that the maps are again the ones indicated above. The lemma then follows.
\end{proof}

Recall that for a ring $R$ with involution we have the stable $\infty$-category $\D^p(R)$ of perfect complexes over $R$. The involution on $R$ induces a canonical duality on $\D^p(R)$, giving rise to homotopy quadratic and homotopy symmetric Poincar\'e structures $\QF^\q$ and $\QF^\s$ which are related by the canonical symmetrisation map $\QF^\q \to \QF^\s$. This map is an equivalence on cross effects. Let us define, for ease of notation, for any space $X$, the visible symmetric and visible quadratic L-theory of $X$ with coefficients in $R$ as follows.
\[ \L^\vs(X;R) = \L((\D^p(R),\QF^\s)_X) \quad \text{ and } \L^{\vq}(X;R) = \L((\D^p(R),\QF^\q)_X).\]
By analysing the linear part of the visible Poincare structure, we find that there is a canonical map of Poincar\'e categories
\[ (\D^p(R)_X,\QF^\q) \to (\D^p(R),\QF^\q)_X \]
is an equivalence, i.e.\ that visible quadratic L-theory is simply quadratic L-theory of the category $\D^p(R)_X$ with its induced duality; we will therefore also write $\L^\q(X;R)$ for $\L^\vq(X;R)$. We note that $\D^p(R)_X \subseteq \D^p(R[\Omega X])$ so that in total we obtain an equivalence $\L^\vq(X;R) \simeq \L^\q_c(R[\Omega X])$ and using the $\pi$-$\pi$-theorem, see e.g.\ \cite[Corollary 1.2.33 (i)]{CDHIII}, even a further equivalence\footnote{under the assumption that $X$ is connected and pointed.} $\L^\q_c(R[\Omega X]) \simeq \L^\q_c(R\pi)$ where $\pi = \pi_1(X)$ and the subscript $c$ stands for an appropriate control, namely the one given by the image of the map $\K_0(R) \to \K_0(R\pi)$.

\begin{Cor}\label[Cor]{Cor:visible-pullback}
The diagram
\[ \begin{tikzcd}
	X\otimes \L^\q(R) \ar[r] \ar[d] & X \otimes \L^\s(R) \ar[d] \\
	\L^\q(X;R) \ar[r] & \L^\vs(X;R)
\end{tikzcd}\]
is a pullback.
\end{Cor}
\begin{proof}
This is a special case of \cref{lemma:general-pullback-L-theories}.
\end{proof}

The following is now the remaining piece in the proof of \cref{Thm:assembly-symmetric-L}.
\begin{Lemma}\label[Lemma]{Lemma:vs=s}
Let $R$ be an involutive ring and $G$ be a 2-torsion free group. Then there is a canonical equivalence
\[ \L^\vs(BG;R) \lto \L^\s_c(RG).\]
Here the subscript $c$ stands for control in the subgroup $\Ima(K_0(R) \to K_0(RG)) \subseteq K_0(RG)$.
\end{Lemma}
\begin{proof}
We first note that the visible symmetric Poincar\'e structure, for connected and pointed $X$, in one on $\D^p(R)_{BG} \subseteq \D^p(RG)$ where the subcategory is the one associated to the subgroup $\Ima(K_0(R) \to K_0(RG))$. Since Poincar\'e structures extend uniquely to idempotent completions, it suffices to argue that on the category $\D^p(RG)$, the visible Poincar\'e structure $\QF^{\vs}$ agrees with the $\QF^\s$. Since the bilinear parts agree, it suffices to compare the linear terms. In this case, we have
\[ \Lambda^{\vs}(M) = \map_{RG}(M, R^{tC_2}) \quad \text{ whereas } \quad \Lambda^\s(M) = \map_{RG}(M, (RG)^{tC_2}) \]
where $RG$ has the $C_2$-action given induced by the involution on $R$ and the inversion action on $G$. The map from left to right is induced by the map $\{e\} \to G$. 
Now, as a module with $C_2$-action, $RG$ therefore decomposes according to the decomposition of $G$ into transitive $C_2$-sets as follows: 
\[ RG = \bigoplus\limits_{g \in G[2]} R \oplus \bigoplus\limits_{[g] \in G \setminus G[2]} \mathrm{ind}_e^{C_2}(R).\]
In particular, if $e$ is the only 2-torsion element in $G$, the the map $R \to RG$ induces an equivalence after applying $(-)^{tC_2}$. Therefore, in this case we find that the canonical map of Poincar\'e structures $\QF^\vs \to \QF^\s$ is an equivalence.
\end{proof}

\begin{proof}[Proof of \cref{Thm:assembly-symmetric-L}]
We consider the following commutative diagram.
\[ \begin{tikzcd}
	BG \otimes \L^\q(R) \ar[r] \ar[d] & \L^\q(RG) \ar[d] \\
	BG \otimes \LL^\q(R) \ar[r] & \LL^\q(RG)
\end{tikzcd}\]
The left vertical map is an equivalence since $\K(R)$ is assumed to be connective. The lower horizontal map is the map which is predicted to be an equivalence by the FJ conjecture. The right vertical map is an equivalence if $\K(RG)$ is connective (though this is \emph{not} and if and only if). Now the K-theoretic FJ conjecture implies that $\K(RG) \simeq BG \otimes \K(R)$ which is again connective by assumption. We conclude that the top horizontal map is an equivalence. 
Now we use the pullback diagram
\[ \begin{tikzcd}
	BG \otimes \L^\q(R) \ar[r] \ar[d] & BG\otimes \L^\s(R) \ar[d] \\
	\L^\q(BG;R) \ar[r] & \L^\vs(BG;R) 
\end{tikzcd}\]
obtained in \cref{Cor:visible-pullback} together with the equivalences $\L^\q(BG;R) \simeq \L^\q(RG)$ (which holds for all groups $G$) and the equivalence $\L^\vs(BG;R) \simeq \L^\s(RG)$ of \cref{Lemma:vs=s} (which holds for 2-torsion free groups $G$). We have argued above that the left vertical map is an equivalence, and therefore so is the right.
\end{proof}

\section{Relations to signature genera}

We now comment on a relation to previous approaches to comparing the Baum--Connes (BC) and Farrell--Jones (FJ) assembly maps and thereby analytic and surgery theoretic approaches to the Novikov conjecture. We recall here that the Novikov conjecture is implied by either of the two assembly maps being rationally injective, and that \cite[Theorem D]{LN} implies that the FJ assembly map is rationally injective if the BC assembly map is rationally injective. In several papers \cite{HR1, HR2, HR3, PS, Wahl} maps from $\L$-theory to $\K(-)\adjt$-theory have been constructed in order to get such a comparison. The idea common to those approaches is to promote the signature operator of an oriented manifold to an appropriate $\K$-theory class. We will review this operator below and connect it to our approach. Note, however, that by \cref{thm} no integral map of spectra from $\L$-theory to $\K$-theory exists and our maps $\tau_\R\colon \ko \to \L\R$ and $\tau_\C\colon \ku \to \L\C$ are indeed maps in the other direction.

\subsection*{The signature operator}
Let us first review the signature operator, see e.g.\ \cite[II.\S6, Example 6.2]{LM} and \cite[Section 1]{RW}. To this end we let $M$ be a closed, oriented, Riemannian manifold of dimension $n$. We consider the de Rham complex 
\[ 
\Omega^*(M;\C) =  \bigoplus_{i=0}^n \Omega^i(M;\C)
\]
of complex valued differential forms on $M$ with the operator $d\colon \Omega^*(M;\C) \to \Omega^*(M;\C)$. The orientation and metric induce inner products on $\Omega^*(M;\C)$ and we denote the formal adjoint of $d$ by $d^*$ as usual. We then consider the elliptic, first order differential operator $D_M := d + d^*$. With respect to the chiral grading defined below $D_M$ is called the \emph{signature operator}. We have that
\[
D_M^2 = D_M^* D_M = dd^* + d^* d =: \Delta_M \ 
\]
is the well known Laplace operator on $M$. Thus the solutions to $D_M = 0$ are given by the solutions to $\Delta_M = 0$ which are the harmonic forms. By Hodge theory, the harmonic forms are isomorphic to $\bigoplus H^*(M; \mathbb{C})$. Now we introduce the chiral $\Z/2$-grading on $\Omega^*(M)$. This is not the grading by even and odd forms (with respect to which the operator $d+d^*$ is the Euler operator whose index is the Euler characteristic). Instead, the grading operator $\tau$ is defined on a $p$-forms $\omega$ by
\[
\tau(\omega) = i^{\lceil n/2 \rceil + p (p+1) + 2p(n-p)} *\omega
\]
where $n = \dim M$ and $*$ is the Hodge-$*$-operator.
It is not hard to check that this is indeed a grading operator, i.e.\ that $\tau^2 = 1$. Also, we note if $n$ is even, then $\tau$ as defined above on $p$-forms $\omega$ satisfies
\[ \tau(\omega) = i^{n/2 + p(p-1)} *\omega.\]
This is the familiar grading operator as for instance considered in \cite[\S 1]{RW} or \cite[II.\S6, Example 6.2]{LM}.
\begin{Rmk}
One can also describe the above using Clifford algebras as follows. One has a canonical (additive) isomorphism $\varphi\colon \Omega^*(M;\C) \xto{\cong} \Gamma(\operatorname{Cliff}_{\C}(TM))$ \cite[I.\S 1, Formula (1.13)]{LM}\footnote{Here, the Clifford algebra bundle is formed using a Riemannian metric on $M$ with the relation $v^2 = - \langle v, v \rangle$ for a tangent vector $v$, i.e. we include the minus sign following \cite{LM}.}. Under this isomorphism the operator $d + d^*$ corresponds to the Dirac operator on the Clifford bundle $\operatorname{Cliff}_{\C}(TM)$, see \cite[II.\S 5, Theorem 5.12]{LM}. The chiral grading defined above then corresponds under this isomorphism to the grading induced by left multiplication with the complex volume element which is the section of $\operatorname{Cliff}_{\C}(TM)$ given locally by $i^{\lceil n/2 \rceil} e_1 \cdots e_n$ for an oriented, orthonormal frame $e_1,..., e_n$, see \cite[Section 1]{RW}. Indeed, we have to check that (locally)
\[
\varphi(\tau(\omega)) = {i^{\lceil n/2 \rceil}} e_1 \cdots e_n \cdot \varphi(\omega)
\]
for $\omega \in \Omega^p(M;\C)_m$ with $m \in M$. 
By rescaling $\omega$ and changing the local frame we may assume that $\omega = e_1^\vee \wedge ... \wedge e_p^\vee$ so that $\varphi(\omega) = e_{1}\cdots e_{p}$. Then we calculate
\begin{align*}
	i^{\lceil n/2 \rceil} e_{1}\cdots e_{n} \cdot \varphi(\omega) & = i^{\lceil n/2 \rceil}e_{1}\cdots e_{n} \cdot e_{1} \cdots e_{p} \\
	& = i^{\lceil n/2 \rceil} (-1)^{(n-1)+\dots+(n-p)}e_{p+1}\cdots e_{n}\cdot e_{1}^{2}\cdots e_{p}^{2} \\
	& = i^{\lceil n/2 \rceil} (-1)^{n+\dots+ n-p+1 - p(n-p) + p(n-p)} e_{p+1} \cdots e_{n} \\
	& = i^{\lceil n/2 \rceil} (-1)^{1+ \dots + p + p(n-p)} e_{p+1} \cdots e_{n} \\ 
	& = i^{\lceil n/2 \rceil} (-1)^{\tfrac{p(p+1)}{2} + p(n-p)} e_{p+1} \cdots e_{n} \\ 
	& = i^{\lceil n/2 \rceil + p(p+1) + 2p(n-p)} e_{p+1} \cdots e_{n} \\
	& = i^{\lceil n/2 \rceil + p(p+1) + 2p(n-p)} \cdot \varphi(e_{p+1}^{\vee} \wedge \dots \wedge e_{n}^{\vee}) \\
	& = \varphi(\tau(\omega))
\end{align*}
as claimed.

\end{Rmk}

If $n$ is even then $\tau$ anticommutes with $D_M$, so that in terms of the decomposition $\Omega^*(M;\C) = \Omega^*(M;\C)^+ \oplus \Omega^*(M;\C)^-$ the operator $D_M$ restricts to $D_M^+: \Omega^*(M;\C)^+ \to \Omega^*(M;\C)^-$. Then we find that the index 
\begin{align*}
\operatorname{ind}(D_M) &=  \dim \ker\big( \Omega^*(M;\C)^+ \stackrel{D_M^+}{\lto} \Omega^*(M;\C)^-\big) - \dim \coker\big( \Omega^*(M;\C)^- \stackrel{D_M^+}{\lto} \Omega^*(M;\C)^+\big) \\
&=  \dim \ker\big(\Omega^*(M;\C)^+ \stackrel{D_M^+}{\lto} \Omega^*(M;\C)^-\big) - \dim \ker\big( \Omega^*(M;\C)^- \stackrel{D_M^-}{\lto} \Omega^*(M;\C)^+\big)
\end{align*}
is given by the signature of the manifold $M$, hence the name signature operator. 
By means of Kasparov's model of the (complex) $K$-homology of $M$ given by $\KU_0(M) \cong \KK(C^0(M), \C)$, we see  that the operator $D_M$ with respect to the chiral grading defines a class in $\KU_0(M)$. In this picture taking the index corresponds to the pushforward $\KU_0(M) \to \KU_0(\mathrm{pt}) = \Z$.

If $n$ is odd, then $\tau$ commutes with $D_M$, and both $\tau$ and $D_M$ anti-commute with the usual even/odd grading operator $\sigma$. We can then consider the operator $D_M$ as graded via the even/odd grading, and use $\tau$ to obtain in addition an action by $\mathrm{Cl}_\C(\R)$ where the odd generator acts via $\tau$. In this way, one obtains the signature operator of $M$ as an element of
\[
\KU_1(M) = \KK(C^0(M), \mathrm{Cl}_\C(\R)) \, 
\]
see \cite[pg.\ 49]{RW}.
Finally we would like to bring the operators just constructed into the top degrees by multiplying with the Bott class. 
To this end we note that we have that 
\begin{align*}
& \KU_0(M^{2n}) \xto{\beta^n} \KU_{2n}(M^{2n}) = \ku_{2n}(M^{2n}) \\
& \KU_1(M^{2n+1}) \xto{\beta^n} \KU_{2n+1}(M^{2n+1}) = \ku_{2n+1}(M^{2n+1})
\end{align*}
where the latter equalities holds since $M$ is $2n$ and $(2n+1)$-dimensional, respectively. 

\begin{definition}\label[definition]{def:signature-class}
For any $n$-dimensional closed oriented manifold $M$ we define the class of the signature operator $[D_M]$ as the class in $\ku_n(M)$ just described.
\end{definition}

One of the main goals of this section is to prove the following result. We denote by $\sigma_\C$ the composite
\[ \MSO \stackrel{\sigma_\R}{\lto} \L(\R) \lto \L(\C) \]
of the Sullivan--Ranicki orientation with the canonical map induced by $\R \to \C$ and by $\tau_\C \colon \ku \to \L(\C)$ the canonical map from \cref{ThmA} or \cite{LN}. Both maps induce maps on homology of $M$:
\[ \MSO_n(M) \stackrel{\sigma_\C}{\lto} \ell(\C)_n(M) \stackrel{\tau_\C}{\longleftarrow} \ku_n(M) \]
again denoted by $\sigma_\C$ and $\tau_\C$, respectively. We denote by $[M] \in \MSO_n(M)$ the bordism class of the identity of $M$.
\begin{Prop}\label[Prop]{Prop_signature}
Let $M$ be an $n$-dimensional closed oriented manifold. In the group $\ell(\C)_n(M)$, we have the equality 
\[
\tau_\C( [D_M]) = 2^{\lfloor n/2 \rfloor}\cdot \sigma_{\C}([M])
\]
up to 2-power torsion, that is, the difference between the two classes is a 2-power torsion element.
\end{Prop}

Before we explain how to prove this statement we would like to ask the following interesting and obvious question:
\begin{problem}
Does the equality of \cref{Prop_signature} hold integrally? 
\end{problem}

The proof of  Proposition \ref{Prop_signature} will proceed in several steps. First note that it is enough to check that the elements agree in $\ell(\C)\adjt_n(M)$ which is what we will in fact do. We first translate the statement into homotopy theory. To this end we would like to understand the signature operator in terms of genera. First recall that for each map of graded rings $\Phi\colon \MSO_* \to R_*$ one can assign a Hirzebruch characteristic series
\[
K_{\Phi}(t) = \frac{t}{\exp_{\Phi}(t)} \in (R_* \otimes \Q)\llbracket t\rrbracket
\] 
where $\exp_{\Phi}(t)$ is the inverse to the logarithm $\log_\Phi(t) = \sum_n \Phi(\C P^n) \frac{t^{n+1}}{(n+1)}$. If we give $t$ degree $-2$ then the Hirzebruch series\footnote{Topologically, the Hirzebruch series is the difference class in $H^0(\C P^\infty, R \otimes H\Q)^\times$ betwen the orientation $\Phi$ and the standard rational orientation of $R \otimes H\Q$.} $K_\Phi(t)$ is of degree $-2$. Note that since $\C P^n$ is nullbordant for odd $n$, the power series we consider here is really power series in $t^2$. We will be interested in the cases  $R_* = \KO_*$ and $R_* = \KU_*$. We can introduce the degree $0$ element $z := \beta t \in \KU_*\llbracket t\rrbracket$ and can rewrite the power series $K_\Phi(t)$ as a power series in $z$:
\[
K_\Phi(z) \in \Q\llbracket z\rrbracket
\]
Even for the case $\KO$ this works since $z^2 = \beta^2 t^2$ exists in the rationalization (recall that $\beta^2 = x/2$, for $x \in \KO_4$ as considered earlier) and the power series is really a series in $z^2$. We will thus also use this convention for $\KO$ and hope this does not lead to confusion. Proposition \ref{Prop_signature} is a consequence of the following more general result, as we will explain below.
\begin{Thm}\label[Thm]{thm_comm_1}
\begin{enumerate}
\item
There is a unique map of $\E_\infty$-rings $\mathcal{L}_{AS}\colon \MSO \to \ko\adjt$ which on homotopy groups induces the map
\[
[M^{4n}] \mapsto 2^{-2n}\beta^{2n}\sign(M^{4n}).\ \footnote{Informally, $\beta^{2n}\sign(M^{4n})$ is the signature of $M$ if $n$ is even and is 2 times the signature if $n$ is odd.} 
\]
\item
For every space $X$, the induced map $\MSO_*(X) \to \ko\adjt_*(X) \xto{c} \ku\adjt_*(X)$ takes a class $[M \xto{f} X]$ to $2^{- \lfloor n/2 \rfloor}  f_*([D_M])$ where $[D_M]$ is the signature class of \cref{def:signature-class}.
\item
The Hirzebruch characteristic series of $\mathcal{L}_{AS}$ is given by
\[
K_{\mathcal{L}_{AS}}(z) = \frac{z/2}{\tanh(z/2)} \ .
\]
\item
We have a commutative diagram
\[
\xymatrix{
\MSO\ar[r]^{ \mathcal{L}_{AS}}\ar[d]_{\sigma_\R} & \ko\adjt \ar[d]^{\tau_\R}\\
\ell(\R) \ar[r]^-{\can} & \ell(\R)\adjt
}
\]
of $\E_\infty$-ring maps. 
\end{enumerate}
\end{Thm}
\begin{Rmk}
Before we prove this theorem, we note that the the genus associated with $\mathcal{L}_{AS}$ is not quite the ordinary signature genus since there are powers of $2$ appearing. In fact, the characteristic series of the ordinary signature genus is $z / \tanh(z)$ by Hirzebruch's signature theorem.
The genus we consider here first came up (to the best of our knowledge) in Atiyah-Singer's deduction of Hirzebruch's signature theorem using their index theorem, see \cite{MR236952}, specifically in equation (6.5) on page 577 in loc.\ cit.\ and the discussion around it for the powers of 2 that appear. Therefore it is not surprising that this genus shows up here. Similar genera have also been considered by Sullivan to construct a version of the orientation $\sigma_\R$, see e.g.\ \cite[\S 5.A]{MR548575}. Therefore, \cref{thm_comm_1} might not come as a surprise to the experts. However, a highly structured statement as Theorem \ref{thm_comm_1} (4) is only possible since we have also constructed the map $\ko\adjt \to \ell(\R)\adjt$ as a map of $\E_\infty$-ring spectra. 
\end{Rmk}

\begin{proof}[Proof of \cref{Prop_signature}]
First, we note that the statement of \cref{Prop_signature} is equivalent to showing the claimed equality after inverting 2.
By (4) of \cref{thm_comm_1}, we also have a commutative diagram of $\E_\infty$-ring spectra
\[ \begin{tikzcd}
	\MSO \ar[r,"c\mathcal{L}_{AS}"] \ar[d,"\sigma_\C"] & \ku\adjt \ar[d,"\tau_\C"] \\
	\ell(\C) \ar[r,"\can"] & \ell(\C)\adjt
\end{tikzcd}\]
Thus, for an $n$-dimensional closed oriented manifold $M$, we have the commutative diagram
\[\begin{tikzcd}
	\MSO_n(M) \ar[r,"c\mathcal{L}_{AS}"] \ar[d,"\sigma_\C"] & \ku_n(M)\adjt \ar[d,"\tau_\C"] \\
	\ell(\C)_n(M) \ar[r,"\can"] & \ell(\C)_n(M)\adjt
\end{tikzcd}\]
By (2) of \cref{thm_comm_1}, the top right composite sends $[M]$ to $2^{-\lfloor n/2 \rfloor}\tau_\C ([D_M])$, so the commutativity of the above diagram indeed gives the equality 
\[ 2^{\lfloor n/2 \rfloor}\cdot \sigma_\C([M]) = \tau_\C([D_M])\]
in $\ell(\C)_n(M)\adjt$ as claimed. 
\end{proof}

\begin{proof}[Proof of \cref{thm_comm_1}]
We first prove the uniqueness statement involved in part (1). In fact we also prove the uniqueness result involving homotopy ring maps. 
To this end we consider the maps 
\begin{align*}
\pi_0\big( \Map_{\E_\infty}(\MSO, \ko\adjt)\big) &\longrightarrow \pi_0\big(\Map_\Sp^{\mathrm{HoRing}}(\MSO, \ko\adjt)\big) \\
& \xto{\pi_*} \Hom_{\mathrm{Ring}}(\MSO_*, \ko\adjt_*) \\
& \xto{K} \Q\llbracket z \rrbracket
\end{align*}
where $\Map^{\mathrm{HoRing}}$ denotes the connected components of the space of maps of spectra that are homotopy ring maps and the first map simply forgets the $\E_\infty$-structure.

Our claim is that all these maps are injective. For the last map $K$ the assertion is true since $\ko_*\adjt$ injects into $\KO_* \otimes \Q$ and rationally, we can reconstruct a genus from its characteristic series. In order to show injectivity of the first two maps it suffices to show that the composites 
\[
\pi_0\big( \Map_{\E_\infty}(\MSO, \ko\adjt)\big) \to \Q\llbracket z\rrbracket \qquad  \pi_0\big(\Map_\Sp^{\mathrm{HoRing}}(\MSO, \ko\adjt)\big) \to \Q\llbracket z\rrbracket 
\]
are injective. This will follow from the theory of orientations developed in \cite{MR0494077} and further in \cite{AHR} as we explain now.
First, we note that we can localize away from $2$ and that the canonical map $\MSpin\adjt \to \MSO\adjt$ is an equivalence. We may therefore replace $\MSO$ above with $\MSpin$.
For any pair of $\E_\infty$-maps $f,g\colon \MSpin \to \ko\adjt$ we have a difference map $f/g\colon \mathrm{bspin} \to \mathrm{gl}_1(\ko\adjt)$ which is a map of spectra. More precisely the space of $\E_\infty$-ring maps $\MSpin \to \ko\adjt$ is a torsor over the space of spectrum maps $\mathrm{bspin} \to \mathrm{gl}_1(\ko\adjt)$. 
Similarly in the case of homotopy ring maps the difference map is an $H$-space map $\mathrm{Bspin} \to \mathrm{Gl}_1(\ko\adjt)$, where $\mathrm{BSpin}$ and $\mathrm{Gl}_1(\ko\adjt)$ denote the infinite loop spaces of the spectra $\mathrm{bspin}$ and $\mathrm{gl}_1(\ko\adjt)$, respectively. 

Moreover from the difference class $f/g$ we can recover the quotient $K_f / K_g$ since $K_f$ was itself constructed rationally as a difference class of $f$ with the standard rational orientation $\MSO \to H\Q \to \ko_\Q$. Thus the whole statement is implied by showing that the canonical maps
\begin{equation}\label{maps-to-investigate}
\pi_0\Map_\Sp(\bspin, \mathrm{gl}_1(\ko\adjt)) \to \pi_0\Map^{H}_{\Spc}(\BSpin, \mathrm{Gl}_1(\ko\adjt)) \to  \pi_0\Map^{H}_{\Spc}(\BSpin, \mathrm{Gl}_1(\ko_\Q))
\end{equation}
are injective. Since $\bspin$ is 3-connected, we note that all mapping spaces do not change when replacing $\gl_1(\ko\adjt)$ with $\tau_{\geq 1}\gl_1(\ko\adjt) \simeq \tau_{\geq1}\gl_1(\ko)\adjt$, and using connectedness of $\bspin$ again, we may replace $\gl_1(\ko\adjt)$ with $\gl_1(\ko)\adjt$ and similarly $\gl_1(\ko_\Q)$ with $\gl_1(\ko)_\Q$. Moreover, since $H$-space maps form a collection of connected components inside the space of all maps, we may also neglect the superscript $H$. The first of the two maps in \eqref{maps-to-investigate} is then induced by the canonical map of spectra $\Sigma^\infty_+ \BSpin \to \bspin$, the counit of the $(\Sigma^\infty_+,\Omega^\infty)$-adjunction. We then consider the following fracture square pullback.
\[ \begin{tikzcd}
	\gl_1(\ko)\adjt \ar[r] \ar[d] & \prod\limits_{p \neq 2} \gl_1(\ko)^\cwedge_p \ar[d] \\
	\gl_1(\ko)_\Q \ar[r] & \Big[ \prod\limits_{p\neq 2} \gl_1(\ko)^\cwedge_p \Big]_\Q
\end{tikzcd}\]
Mapping $\bspin$ and $\Sigma^\infty_+ \BSpin$ into this pullback, we obtain pullback descriptions for both mapping spaces in question. We then observe that 
\[\pi_1(\Map_\Sp(\bspin,\Big[ \prod\limits_{p\neq 2} \gl_1(\ko)^\cwedge_p \Big]_\Q)) = 0 = \pi_1(\Map_\Sp(\Sigma^\infty_+\BSpin,\Big[ \prod\limits_{p\neq 2} \gl_1(\ko)^\cwedge_p \Big]_\Q)) \]
This follows simply from the observation that the homotopy groups of the rational spectrum $\big[ \prod\limits_{p\neq 2} \gl_1(\ko)^\cwedge_p \big]_\Q$ are concentrated in degrees $4k$, the rational homotopy of $\bspin$ is in degrees $4k$ and $\BSpin$ has rational cohomology also concentrated in degrees $4k$. We deduce that for $X = \bspin$ and $X= \Sigma^\infty_+ \BSpin$, the canonical map
\[ \pi_0\Map_\Sp(X,\gl_1(\ko)\adjt) \lto \pi_0\Map_\Sp(X,\gl_1(\ko)_\Q) \times \prod\limits_{p\neq 2} \pi_0\Map_{\Sp}(X,\gl_1(\ko)^\cwedge_p) \]
is injective. It therefore suffices to argue that the map $\Sigma^\infty_+ \BSpin \to \bspin$ induces a $\pi_0$ injection upon mapping to $\gl_1(\ko)_\Q$ and $\gl_1(\ko)^\cwedge_p$ for all $p\neq 2$ individually. Note that $\gl_1(\ko)_\Q = \prod\limits_{k\geq 1} H\Q[4k]$, so that it suffices to know that the map $\Sigma^\infty_+\Omega^\infty X \to X$ induces an injection on rational cohomology in all degrees, for all connective spectra $X$. To treat the $p$-adic case, we recall from \cite[Theorem 4.11]{AHR}, applied to $\KO^\cwedge_p$, that $\gl_1(\KO)^\cwedge_p \to L_{K(1)} \gl_1(\KO) \simeq \KO^\cwedge_p$ is 1-truncated. Using again that $\bspin$ is 3-connected, it suffices now to show that the map $\Sigma^\infty_+\BSpin \to \bspin$ induces a $\pi_0$-injection on mapping spaces to $\KO^\cwedge_p$. Since this spectrum is $K(1)$-local, it finally suffices to note that the map $\Sigma^\infty_+\BSpin \to \bspin$ has a $K(1)$-local section. This is a consequence of the fact that $L_{K(1)} = \Phi\Omega^\infty$, where $\Phi$ is the Bousfield--Kuhn functor, see \cite[proof of Prop.\ 2.9]{LMMT} for details. This shows that the first map of \eqref{maps-to-investigate} is injective as claimed.

We turn to the second map of \eqref{maps-to-investigate}. Using again the fracture square for $\gl_1(\ko)\adjt$ as before, the statement follows if we can show that for $F$ the fibre of the map
\[ \prod\limits_{p\neq 2} \gl_1(\ko)^\cwedge_p \lto \Big[ \prod\limits_{p \neq 2} \gl_1(\ko)^\cwedge_p\Big]_\Q \]
we have 
\[ \pi_0(\Map_\Sp(\Sigma^\infty_+\BSpin,F)) = 0.\]
With a similar argument as before, using again that $\BSpin$ is 3-connected, we may equivalently replace $F$ by the fibre of the map 
\[ \prod\limits_{p \neq 2} \KO^\cwedge_p \lto \Big[ \prod\limits_{p \neq 2} \KO^\cwedge_p \Big]_\Q \]
which is, again by a fracture square argument the same fibre as that of the map $\KO\adjt \to \KO_\Q$. Since this map is a retract of $\KU\adjt \to \KU_\Q$, it finally suffices to show that \[\pi_0(\Map_\Sp(\Sigma^\infty_+\BSpin, \mathrm{fib}(\KU\adjt \to \KU_\Q)))=0.\]
Since $\BSpin$ has even rational cohomology, this is equivalent to showing that the map 
\[ \pi_0(\Map_\Sp(\Sigma^\infty_+ \BSpin,\KU\adjt)) \lto \pi_0(\Map_\Sp(\Sigma^\infty_+\BSpin,\KU_\Q)) \]
is injective. For this, we recall that $\BSpin \simeq \colim_n \BSpin(n)$ 
and that $\Spin(n)$ is a compact connected Lie group. In \cite{Anderson}, Anderson shows that for any compact connected Lie group $G$, we have that $\KU\otimes BG$ is a filtered colimit of direct sums of $\KU$, with split injective transition maps. Consequently, we find that for any $\KU$-module $M$, we have
\[ \Map_\Sp(\Sigma^\infty_+ \BSpin(n),M) = \Map_\KU(\colim\limits_{i\in I} \bigoplus\limits_{A_i} \KU,M) = \lim\limits_{i \in I} \prod\limits_{A_i} M\]
naturally in $M$. In particular, the map 
\[ \Map_\Sp(\Sigma^\infty_+\BSpin(n),\KU\adjt) \lto \Map_\Sp(\Sigma^\infty_+\BSpin(n),\KU_\Q) \]
identifies with the map
\[ \lim\limits_{i\in I} \prod\limits_{i \in A_i} \KU\adjt \lto \lim\limits_{i \in I} \prod\limits_{i \in A_i} \KU_\Q\]
which is injective on $\pi_0$. Finally, the map we wish to show is injective identifies with the map on $\pi_0$ induced by the map
\[ \lim\limits_n \Map_\Sp(\Sigma^\infty_+\BSpin(n),\KU\adjt) \lto \lim\limits_n \Map_\Sp(\Sigma^\infty_+\BSpin(n),\KU_\Q).\]
Now, since for each $n$ both mapping spaces which appear have no odd homotopy groups, the $\lim$-$\lim^1$-sequence shows that the map we wish to show is injective identifies with the map
\[ \lim\limits_n \pi_0\Map_\Sp(\Sigma^\infty_+\BSpin(n),\KU\adjt) \lto \lim\limits_n \pi_0\Map_\Sp(\Sigma^\infty_+\BSpin(n),\KU_\Q).\]
This is an inverse limit of injective maps, and hence itself injective as claimed.
In particular, we have shown that there is at most one $\E_\infty$-map as in (1).

For the existence of this map, Ando-Hopkins-Rezk \cite[Theorem 6.1]{AHR} give a concrete criterion in terms of certain $p$-adic congruences, see also \cite[Theorem 3.1.1]{Sprang}. One could simply verify these directly, which is for instance done by Wilson in \cite[Theorem 5.5]{Wilson}. Instead of using this calculation, we will proceed differently and simply use the square of part (4) as a proof of the existence of an $\E_\infty$-map with the correct effect on homotopy groups, since the right vertical map in it is an equivalence. This then also shows the commutativity of (4) immediately. 

In order to prove statement (2), we use that the assignment 
\[
\MSO_n(X) \lto \ku\adjt_n(X) \qquad (M \xto{f} X) \mapsto 2^{- \lfloor n/2 \rfloor}  f_*([D_M])
\]
is a map of multiplicative cohomology theories as shown in \cite{RW}, specifically see Remark 4 and Lemma 6 in loc.\ cit.
Thus by the previous results it is enough to check that it agrees with the map of part (1) on coeffcients, which is true by construction. 

For (3) we want to compute the Hirzebruch series of the map $\mathcal{L}_{AS}$. 
We find that
\[
\mathcal{L}_{AS}(\C P^n) = \begin{cases}
2^{-n} \beta^n & n \text{ even} \\
0 & n \text{ odd}
\end{cases}
\]
where $\beta = \beta_\C$ is the complex Bott element. 
Thus we get that 
\begin{align*}
\log_{\mathcal{L}_{AS}}(t) &= t + \tfrac{\beta^2}{2^2} \cdot \tfrac{t^3} 3 +  \tfrac{\beta^4}{2^4} \cdot \tfrac{t^5} 5 + \tfrac{\beta^6}{2^6} \cdot \tfrac{t^7} 7 + ...  \\
 & = \tfrac 2 \beta \cdot  ((\beta t /2) + \tfrac{(\beta t /2)^3} 3 +  \tfrac{(\beta t /2)^5} 5 + \tfrac{(\beta t/2)^7} 7 + ...) \\
 & =  \tfrac 2 \beta  \tanh^{-1}(\beta t/2) 
\end{align*}
The inverse of this power series (with respect to composition) is given by
\[ \exp_{\mathcal{L}_{AS}}(t) = \tfrac 2 \beta \tanh(\beta t /2) \]
as one directly verifies. Therefore we get 
\[ K_{\mathcal{L}_{AS}}(t) = \frac{ \beta t / 2 } { \tanh(\beta t /2)} =  \frac{ z/2 } { \tanh(z/2)} \]
where we recall that $z = \beta t$.
\end{proof}

\begin{Rmk}\label[Rmk]{Rmk:Warning}
In \cite[Theorem 5.7]{Wilson}, Wilson writes that in addition to the map $\mathcal{L}_{AS}\colon \MSpin \to \ko\adjt$ described above, there also exists an integral $\E_\infty$-map $\mathcal{L}_H\colon \MSpin \to \ko$ sending a spin manifold $M^{4n}$ to $\beta^{2n}\sign(M)$. This, however, is not correct, and the map $\mathcal{L}_H$, as an $\E_\infty$-map or equivalently by proof of \cref{thm_comm_1} as a map of homotopy ring spectra, indeed only exists after inverting 2. The fact that it does exist after inverting 2 can be shown using the criterion of Ando--Hopkins--Rezk  \cite[Theorem 6.1]{AHR}, or by postcomposing $\mathcal{L}_{AS}$ with the Adams operation $\psi^2 \colon \ko\adjt \to \ko\adjt$. We thank Johannes Sprang for explaining to us the following argument that the map does not exist at {the} prime 2. To explain this, we recall again the general result of Ando--Hopkins--Rezk: It says that the connected components of the space of $\E_\infty$-maps $\MSpin \to \ko$ are in bijection to the set of sequences\footnote{For an $\E_\infty$-map $\MSpin \to \ko$, this sequence is given by the coefficients in the characteristic series as described above.} $(b_k)_{k\geq 2} \in \mathbb Q$ satisfying the following conditions:
\begin{enumerate}
\item $b_{2k+1} = 0$ for $k\geq 1$,
\item $b_{2k} \equiv - \frac{B_{2k}}{2k} \mod \mathbb Z$, and
\item for every prime $p$ and every element $c \in \Z_p^\times/\{\pm 1\}$, there exists a $p$-adic measure $\mu$ on $\Z_p^\times/\{\pm 1\}$ such that for all $k \geq 1$ one has\footnote{We remark that $(1-c^{2k})b_{2k}$ is a $p$-adic integer.}
\[ (1-p^{2k-1})(1-c^{2k})b_{2k} = \int_{\Z_p^\times/\{\pm 1\}} x^{2k}d\mu(x),\]
see \cite[Definition 2.1.4]{Sprang} for the notion of $p$-adic measures on profinite groups such as $\Z_p^\times/\{\pm 1\}$.
\end{enumerate}
The sequence relevant for realising the map sending $M^{4n}$ to $\beta^{2n}\sign(M)$ as an $\E_\infty$-map $\MSpin \to \ko$ is the sequence
\[ b_{k} := \tfrac{2^{k+1}(2^{k-1}-1)}{2k}B_k.\]
Conditions (1) and (2) above are indeed satisfied, see \cite[Proposition A.3]{Wilson} for (2) and (1) follows from the same property for the Bernoulli numbers $B_{k}$. Now, at prime 2, one observes that the sequence
\[ (1-2^{2k-1})(1-c^{2k})b_{2k} \]
converges to zero in $\Z_2$. However, a sequence of moments, i.e.\ a sequence of the form 
\[ \int_{\Z_p^\times/\{\pm 1\}} x^{2k}d\mu(x) \]
converges to zero only if it is constantly zero. Indeed, $x\mapsto x^{2k+2^r} - x^{2k}$ for $x \in \Z_2^\times$ takes values in $2^r \Z_2$. Consequently, 
\[ |\int_{\Z_p^\times/\{\pm 1\}} x^{2k} d\mu(x)|_2 = \lim_{r\to \infty} |\int_{\Z_p^\times/\{\pm1\}} x^{2k+\phi(2^r)}|_2 \]
where $|-|_2$ denotes the 2-adic valuation, and the latter term is zero if we assume that the sequence of moments converges to zero. Now, the sequence we need to investigate converges to zero, but is not constantly zero, and is therefore not a sequence of moments.
\end{Rmk}

We finish this section by noting that  there is a commutative diagram

\[ \begin{tikzcd}
	\MSO \ar[r,"\mathcal{L}_H"] \ar[d,"\sigma_\R"] \ar[rr, bend left, "\mathcal{L}_{AS}"] & \ko\adjt \ar[r,"\psi^{-2}"] \ar[d,"\tau_\R"] & \ko\adjt \ar[d,"\tau_\R"] \\
	\ell(\R) \ar[r,"\alpha"] \ar[rr, bend right, "\mathrm{can}"] & \ell(\R)\adjt \ar[r,"\psi^{-2}"] & \ell(\R)\adjt
\end{tikzcd}\]
where we denote by $\psi^{-2}$ also the induced (inverse) Adams operation on $\ell(\R)\adjt$. The resulting map $\alpha$ is then given by the right-down composite in the diagram
\[ \begin{tikzcd}
	\ell(\R) \ar[r,"\mathrm{can}"] \ar[d, dashed,"\psi^2"] & \ell(\R)\adjt \ar[d,"\psi^2"] \\
	\ell(\R) \ar[r,"\mathrm{can}"] & \ell(\R)\adjt
\end{tikzcd}\]

\begin{question}
Does there exist an $\E_\infty$-map $\psi^2\colon \ell(\R) \to \ell(\R)$ rendering the above diagram commutative? Likewise, does there exist an $\E_\infty$-map $\psi^2 \colon \ell(\C) \to \ell(\C)$ rendering the analogous diagram
\[ \begin{tikzcd}
	\ell(\C) \ar[r,"\mathrm{can}"] \ar[d, dashed,"\psi^2"] & \ell(\C)\adjt \ar[d,"\psi^2"] \\
	\ell(\C) \ar[r,"\mathrm{can}"] & \ell(\C)\adjt
\end{tikzcd}\]
commutative, where we use $\tau_\C \colon \ku\adjt \stackrel{\simeq}{\to} \ell(\C)\adjt$ to define $\psi^2$.
\end{question}

One can show that the map $\psi^2$ (in both the real and the complex case) exists as a map of $\E_1$-algebras. In order to construct this, one can use that $\ell(\R)$ and $\ell(\C)$ are 2-locally the free $\E_1$-$H\Z$-algebra on a generator in degree 4 and 2, respectively, see also \cite[Corollary 4.2]{HLN}. Then, the map $\psi^2$ is constructed as to be an $H\Z$-algebra map at prime 2. At the time of writing, we do not know whether $\ell(\R)$ or $\ell(\C)$ are 2-locally $\E_\infty$-$H\Z$-algebras, and in addition, should this be the case, we do not know whether to expect a possible $\E_\infty$-map $\psi^2 \colon \ell(\R) \to \ell(\R)$ to be 2-locally a map of $\E_\infty$-$H\Z$-algebras.

\begin{Rmk}
A curious consequence of the existence of the $\E_1$-map $\psi^2 \colon \ell(\C) \to \ell(\C)$ is the following observation about formal groups. We recall that the formal group of $\ku$ is the multiplicative one, in particular $\ku$ has a coordinate given by $x+y+\beta_\C xy$. The map $\tau_\C \colon \ku \to \ell(\C)$ provides a coordinate of the formal group of $\ell(\C)$ which is then given by $x+y+2b_\C xy$, where $b_\C \in \L_2(\C)$ is the periodicity generator, since $\tau_\C(\beta_\C) = 2b_\C$. Postcomposition with powers of $\psi^2$ on $\ell(\C)$ gives another coordinate of the formal group of $\ell(\C)$ given by $x+y+2^kb_\C xy$, with $k\geq 1$. As any two coordinates of a formal group are connected by a (strict) isomorphism, we deduce that for $k\geq 1$, the formal group laws $x+y+2xy$ and $x+y+2^kxy$ are isomorphic over $\Z$.
\end{Rmk}

\section{Further remarks}

\subsection*{On maps between K-theory and L-theory}
In this subsection, we aim to analyse, similarly to \cite{LN} the possible integral maps between K- and L-theory.
Let us first consider the map $\ko \to \L(\R)$ and describe its effect on homotopy groups.
For this, and in general, it will be convenient to record the following result:

\begin{Lemma}
The transformation $\tau$ of \cref{ThmA} is compatible with the unique lax symmetric monoidal transformation $\tau$ of \cite[Theorem A]{LN} in the sense that there is a commutative diagram of lax symmetric monoidal functors 
\[ \begin{tikzcd} \k(-) \ar[r, "\tau"] \ar[d] & \L(-) \ar[d] \\ \k(-\otimes \C) \ar[r, "\tau"] & \L(-\otimes\C).\end{tikzcd} \]
\end{Lemma}
\begin{proof}
One observes that the complexification functor sending $A$ to $A_\C = A \otimes_\R \C$ from $C^*$-algebras to complex $C^*$-algebras descends to a symmetric monoidal functor
\[ (-)\otimes\C \colon \KK_\R \lto \KK.\]
Then both composites of the diagram in question are lax symmetric monoidal transformations from $\k(-) \to \L(-)$. Using again that $\k(-)$ is initial, there is (up to canonical equivalence) only one such transformation.
Spelling this out explicitly, we obtain for each $A \in \RAlg$ a commutative diagram
\[ \begin{tikzcd} \k(A) \ar[r] \ar[d] & \L(A) \ar[d] \\ \k(A_\C) \ar[r] & \L(A_\C) \end{tikzcd}\]
which is natural in $A$.
\end{proof}

\begin{example}\label[example]{effect on homotopy}
Applying this in the case $A=\R$, we in particular obtain a commutative diagram of $\E_\infty$-ring spectra given by 
\[ \begin{tikzcd} \ko \ar[r,"\tau"] \ar[d] & \L\R \ar[d] \\ \ku \ar[r,"\tau"] & \L\C \end{tikzcd} \]
and the map induced on homotopy rings of the lower horizontal map is given by 
\[ \Z[\beta] \lto \Z[b_\C] \]
sending $\beta$ to $2b_\C$, see \cite[Lemma 4.9]{LN}. Using this, we can again describe the map
$\ko \to \L\R$ on homotopy rings as follows: Firstly, recall that 
\[ \pi_*(\ko) = \Z[\eta,x,\beta_\R]/(\eta^3,2\eta,\eta x, x^2 = 4\beta_\R) \]
with $|\eta| = 1$, $|x| = 4$ and $|\beta_\R| = 8$. The map $\ko \to \ku$ vanishes on $\eta$, sends $\beta_\R$ to $\beta^4_\C$ and $x$ to $2\beta^2_\C$.
On homotopy, the map $\L\R \to \L\C$ identifies with the canonical inclusion
\[ \Z[b^2_\C] \subseteq \Z[b_\C] \]
as the subring generated by $b^2_\C$. We denote the element in $\L_4(\R)$ corresponding to $b_\C^2$ by $b$. It then follows that the map $\ko \to \L\R$ sends $x$ to $8b$ and $\beta_\R$ to $16b^2$. Notice that this is indeed compatible with the ring structure of $\pi_*(\ko)/\mathrm{torsion}$ and our general analysis as in \cref{Rem:induced-map}.
\end{example}

Just like in the complex case, the only possibility for an interesting integral map between K-theory and L-theory is the one just constructed. More precisely, the analog of \cite[Theorem E]{LN} in the real case holds as well:
\begin{Thm}\label[Thm]{thm}
We have that 
\[ [\KO,\L\R] = 0 = [\L\R,\KO] = [ \ell\R,\KO],\]
where the square brackets denote homotopy classes of maps of spectra.
\end{Thm}
\begin{proof}
The main ingredients in proving this result in the complex case are
\begin{enumerate}
\item[-] Both $\KU$ and $\L\C$ are Anderson self-dual,
\item[-] the map $\KU\otimes \L\C \to (\KU\otimes \L\C)\adjt$ is an equivalence, and
\item[-] the spectrum $\KU\otimes \L\C$ is even, i.e.\ has no odd homotopy groups.
\end{enumerate}
The analog of these results hold true for $\KO$ in place of $\KU$ and $\L\R$ in place of $\L\C$ because
\begin{enumerate}
\item[-] $I_\Z(\KO) \simeq \Sigma^4 \KO$, see \cite[Theorem 8.1]{HS} and $I_\Z(\L\R) \simeq \L\R$ simply because the homotopy groups of $I_\Z(\L\R)$ are again free of rank 1 over the homotopy groups of $\L\R$, just like for $\L\C$.
\item[-] to show that $2$ is invertible in $\KO\otimes \L\R$ it suffices to observe that $\KU \simeq \mathrm{cofib}(\eta\colon \Sigma\KO \to \KO)$, and hence for any spectrum $E$ in which $\eta$ is trivial (such as $\L\R$), we have 
\[ \KU \otimes E \simeq \KO\otimes E \oplus \Sigma^2\KO\otimes E.\]
It follows that $\KO \otimes \L\R$ is a direct summand in $\KU \otimes \L\R$ which is itself a direct summand of $\KU \otimes \L\C$, as $\L\R\oplus \Sigma^2 \L\R \simeq \L\C$. 
\item[-] We have just established that $\KO\otimes \L\R$ is a direct summand of $\KU\otimes \L\C$, so is even as well. 
\end{enumerate}

\end{proof}

\begin{Rmk}
We also remark that, as expected, $\ell(\R)$ is not a compact $\ko$-module\footnote{I.e.\ the functor $\map_\ko(\ell(\R),-)$ does not preserve filtered colimits. Equivalently, $\ell(\R)$ is not perfect, that is, it is not in the stable subcategory of $\ko$-modules generated by $\ko$ under finite colimits, shifts, and retracts.}, and likewise that $\ell(\C)$ is not a compact $\ku$-module. Indeed, it suffices to show the latter, as 
\[ \ku \otimes_\ko \ell(\R) \simeq \ell(\C) \]
so if $\ell(\C)$ is not compact over $\ku$, then $\ell(\R)$ is also not compact over $\ko$. To show this, we observe that $\ell(\C) \otimes_\ku \KU = \KU\adjt$ is obtained from $\ell(\C)$ by inverting $2b$ and $\L(\C)\adjt \simeq \KU\adjt$. Now, $\KU\adjt$ is not compact over $\KU$. Indeed, if it were compact we would have
\[ \KU\adjt \simeq \map_\KU(\KU\adjt,\KU\adjt) = \colim \map_{\KU}(\KU\adjt,\KU) \]
but this colimit is constant, as $2$ is already invertible on the mapping spectrum. The latter is therefore equivalent to $\lim \KU$ with transition maps given by the multiplication by $2$ map. But we have $\pi_0(\lim \KU) = 0$ by the Milnor sequence.
\end{Rmk}

\bibliographystyle{amsalpha}
\bibliography{mybib}

\end{document}